\documentclass[12pt]{article}

\usepackage[english]{babel}
\usepackage[utf8x]{inputenc}		
\usepackage{enumerate}
\usepackage{textcomp}               
\usepackage{typearea}
\usepackage{pdfpages}
\usepackage[toc]{appendix}
\usepackage[]{todonotes}
\usepackage{subcaption}
\usepackage{verbatim}
\usepackage{comment}
\usepackage{mathtools}
\mathtoolsset{showonlyrefs,showmanualtags}
\usepackage[hidelinks]{hyperref}

\usepackage{enumitem}
\usepackage{amsfonts, amsmath, amssymb, amsthm, dsfont}
\usepackage{amsthm}			
\usepackage{thumbpdf}	
\usepackage{natbib}
\allowdisplaybreaks

\usepackage{booktabs}

\usepackage[noblocks]{authblk}
\usepackage{ulem}

\usepackage{pgf, tikz}
\usepackage{tikz-cd}		
\usepackage{bbm}

\theoremstyle{definition}
\newtheorem{definition}{Definition}[section]

\newtheorem*{assumption*}{Assumption}

\newtheorem*{condition*}{Condition}

\theoremstyle{plain}
\newtheorem{theorem}[definition]{Theorem}
\newtheorem{proposition}[definition]{Proposition}
\newtheorem{lemma}[definition]{Lemma}
\newtheorem{corollary}[definition]{Corollary}
\theoremstyle{remark}
\newtheorem{remark}[definition]{Remark}

\makeatletter
\patchcmd{\proof@init}{\@empty}{Proof\@ifnotempty{#1}{\ #1}}{}{}
\makeatother

\usepackage{graphicx}

\newcommand{\N}{\mathbb{N}}

\newcommand{\R}{\mathbb{R}}

\newcommand{\E}{\mathbb{E}}

\newcommand{\F}{\mathcal{F}}

\newcommand{\pconv}{\xrightarrow{\mathbb{P}}}

\newcommand{\Var}{\operatorname{Var}}
\newcommand{\Cov}{\operatorname{Cov}}

\newcommand{\deq}{\overset{d}{=}}

\newcommand{\wconv}{\Rightarrow}
\newcommand{\iid}{\overset{iid}{\sim}}

\newcommand{\Fhyp}{\mathbb{F}}
\newcommand{\NSP}{\textsc{NSP}}

\newcommand{\beps}{\boldsymbol{\epsilon}}

\excludecomment{proofcomplete}
\excludecomment{comment}

\title{At the edge of Donsker's Theorem:\\ Asymptotics of multiscale scan statistics} 

\author[1]{Johann Köhne}
\author[2]{Fabian Mies}

\affil[1]{\footnotesize University of Göttingen}
\affil[2]{\footnotesize Delft University of Technology}

\date{}

\begin{document}

\maketitle

\begin{abstract}
\noindent

For nonparametric inference about a function, multiscale testing procedures resolve the need for bandwidth selection and achieve asymptotically optimal detection performance against a broad range of alternatives. 
However, critical values strongly depend on the noise distribution, and we argue that existing methods are either statistically infeasible, or asymptotically sub-optimal. 
To address this methodological challenge, we show how to develop a feasible multiscale test via weak convergence arguments, by replacing the additive multiscale penalty with a multiplicative weighting.
This new theoretical foundation preserves the optimal detection properties of multiscale tests and extends their applicability to nonstationary nonlinear time series via a tailored bootstrap scheme.
Inference for signal discovery, goodness-of-fit testing of regression functions, and multiple changepoint detection is studied in detail, and we apply the new methodology to analyze the April 2025 power blackout on the Iberian peninsula.
Our methodology is enabled by a novel functional central limit in Hölder spaces with critical modulus of continuity, where Donsker's theorem fails to hold due to lack of tightness. Probabilistically, we discover a novel form of thresholded weak convergence that holds only in the upper support of the distribution.

\end{abstract}

\section{Introduction}

The fundamental problem of signal discovery, or anomaly detection, may be modeled as 
\begin{align}
    Y_t = f_n(t)+ \eta_t, \quad t=1,\ldots, n, \label{eqn:model-signal}
\end{align}
for iid centered random variables $\eta_t$ with variance $\sigma^2=\Var(\eta_t)$, and a regression function of the form $f_n(t)=\mu_n \mathds{1}(a_n < t \leq b_n)$. 
The statistical problem is to detect if a signal is present, that is, to test the null hypothesis $\mathbb{H}_0:\mu_n=0$ versus $\mathbb{H}_1:\mu_n\neq 0$. 
An established statistical procedure, both in theory and practice \citep{glaz_scan_2009}, is based on the local scan statistic $T_n(I)= |\sum_{t\in I} Y_t| / \sqrt{|I|}$ for any interval $I=(a,b]$ with $0\leq a<b\leq n$. If $a,b\notin \N$, the sum may be extended by linear interpolation. 
The null hypothesis is rejected for large values of the global scan statistic $T^{\text{SCAN}}_n=\max_I T_n(I)$. 
As concisely reviewed by \citet{walther_calibrating_2022}, the statistic $T^{\text{SCAN}}_n$ is dominated by the small intervals $I$, and accordingly has an extreme value limit distribution of Gumbel type for standard Gaussian errors $\eta_t\sim \mathcal{N}(0,1)$ \citep{sharpnack_exact_2016}.
In particular, $T^{\text{SCAN}}_n=\sqrt{2\log n} + o_P(1)$, and thus the signal with size $\mu_n$ and length $l_n=b_n-a_n$ can only be detected consistently if $|\mu_n|^2 l_n > 2\log n$. 
Due to the dominance of short intervals, this threshold is asymptotically suboptimal for longer signals of length $l_n\propto n$. In fact, \cite{dumbgen_multiscale_2001} show that a signal may be consistently detected if $\mu_n^2 l_n > 2 \log(\frac{en}{l_n})$, which is achieved by a multiscale statistic of the form
\begin{align}
    T^{\text{DS}}_n = \max_I \; \left\{T_n(I)-\sigma\sqrt{2 \log \frac{en}{|I|}} \right\}_+,\label{eqn:DS} \tag{DS}
\end{align}
where $x_+ = \max(x,0)$.
The asymptotic distribution of $T^{\text{DS}}_n$ with Gaussian errors is accurately described by \cite{dumbgen_multiscale_2001} and critical values may be computed. 
The concept of multiscale test statistics has since found multiple further applications, in particular in changepoint inference \citep{frick2014multiscale,fryzlewicz_narrowest_2024,bastian_multiscale_2025}, shape inference for a density from direct measurements \citep{dumbgen_multiscale_2008} or via deconvolution \citep{schmidt-hieber_multiscale_2013}, and rank-based tests \citep{rohde_adaptive_2008}.

The additive penalty for the multiscale statistic and the corresponding critical values are only valid for Gaussian noise.
The latter setting is often considered as prototypical, as methods developed for the the Gaussian case can usually be applied to non-Gaussian data by virtue of the central limit theorem.
However, for the multiscale statistic $T_n^{\text{DS}}$, the situation is more involved: The following proposition shows that if the noise has only slightly heavier, but still sub-Gaussian tails, the statistical inference is invalidated.
\begin{proposition}\label{prop:invalid}
    Let $\mu_n=0$, i.e.\ no signal, and $\eta_t = \epsilon_t Z_t$ for $Z_t\sim \mathcal{N}(0,1/p)$ and $\epsilon_t\sim \text{Bin}(1,p)$ for some $p\in(0,1)$, such that $\Var(\eta_t)=1$. 
    Then $T_n^{\mathrm{DS}}\to \infty$ in probability.
\end{proposition}
This finding is a direct consequence of Proposition \ref{prop:DS-A} below.
It reveals that for multiscale testing, one can not directly rely on asymptotic arguments to reduce the non-Gaussian case to the Gaussian. 
As a remedy, it has been suggested to impose a lower bound $|I|>h_n$ on the length of the considered intervals in \eqref{eqn:DS} which accounts for the speed of convergence of the central limit theorem for $T_n(I)$, and study the statistic $T_{n, h_n}^{\text{DS}} = \max_{|I|\geq h_n} \{T_n(I) - \sigma \sqrt{2\log(en/|I|)}\}_+$.
For sub-exponential errors, \citet{schmidt-hieber_multiscale_2013} and \citet{frick2014multiscale} impose $h_n \gg \log^{3}(n)$, and \cite{konig_multidimensional_2020} require $h_n\gg \log(n)^{12}$.
For auto-correlated errors with polynomial tails, \cite{dette_multiscale_2020} and \cite{khismatullina_multiscale_2020} require $h_n \gg n^{q}$ for some exponent $q\in(0,1)$ depending on the order of the tail bound and the decay of temporal dependence.
As Proposition \ref{prop:invalid} shows, choosing $h_n$ too small will invalidate the statistical analysis, but this theoretical lower bound is not known in practice. 
Thus, $h_n$ will necessarily be chosen too big, sacrificing power against short signals. 
Even in the idealized situation where a sharp and admissible $h_n$ is known, the test lacks power in the regime $l_n \ll h_n$.
We show that the loss of power can be expressed in terms of the ratio $l_n/h_n$.

\begin{proposition}\label{prop:powerless}
    Suppose that $\eta_t$ are iid standard Gaussian, and $l_n \leq h_n$. 
    If $\mu_n^2 l_n (\frac{l_n}{h_n}) \gg \log(\frac{en}{h_n})$, then $T_{n,h_n}^{\text{DS}} \to \infty$ and the test is consistent. 
    On the other hand, if $\mu_n^2 l_n (\frac{l_n}{h_n})= \mathcal{O}(1)$, then $T_{n, h_n}^{\text{DS}} = \mathcal{O}_P(1)$ and the test fails to be consistent.
\end{proposition}

To summarize, the multiscale statistic $T_n^{\text{DS}}$ of \cite{dumbgen_multiscale_2001} does not maintain the size for non-Gaussian errors, and the truncated version $T_{n,h_n}^{\text{DS}}$ is suboptimal for short signals.

The contribution of this paper is to develop a broadly applicable asymptotic theory for multiscale statistics of non-Gaussian data, giving rise to a feasible and optimal multiscale testing procedure in a wide range of sampling settings.
Instead of the additively penalized statistic $T^{\text{DS}}_n$, we suggest to use the multiplicatively weighted statistic
\begin{align}
    T_n^* = \max_I \frac{T_n(I)}{\sqrt{\log{e\frac{n}{|I|}}}}. \label{eqn:Tstat}
\end{align}
We show that $T_n^*/\sigma$ is asymptotically pivotal under the assumption of sub-Gaussian errors.
Importantly, the statistical methodology is agnostic of the exact tail bound and only uses the noise variance, which can be estimated reliably.
Our probabilistic results are based on the analysis of edge cases in Donsker's theorem in Hölder-type spaces.
Introducing the interpolated partial sum process $S_n(u) = \frac{1}{\sqrt{n}} \sum_{t=1}^{\lfloor un\rfloor} Y_t + \frac{un-\lfloor un\rfloor}{\sqrt{n}} Y_{\lfloor un\rfloor +1}$, and the modulus of continuity $\rho_2(h) = \sqrt{h \log\frac{e}{h}}, h\in(0,1)$, the statistic $T_n^*$ may be expressed as the Hölder-type seminorm
\begin{align}
    T_n^* = |S_n|_{\rho_2} = \sup_{u,v\in[0,1]} \frac{|S_n(u)-S_n(v)|}{\rho_2(|u-v|)}. \label{eqn:Tn-holder}
\end{align}
Donsker's theorem \citep{Billingsley1999} establishes weak convergence of $S_n(u)$ towards a Brownian motion $\sigma B(u)$ in $C[0,1]$, whereas the treatment of \eqref{eqn:Tn-holder} requires weak convergence in the stronger Hölder-type space $C^{\rho_2}$. 
In view of Donsker's theorem in Hölder spaces \citep{lamperti_convergence_1962,hamadouche_weak_1998,rackauskas_necessary_2004,rackauskas_necessary_2004-1,rackauskas_convergence_2020}, the modulus $\rho_2$ presents exactly the edge of its applicability: the limiting Brownian motion is stochastically bounded in $C^{\rho_2}$, but not tight. 
Thus, Prokhorov's Theorem can not be used to show weak convergence in this space, which indeed does not even hold.
Surprisingly, it is possible to show that $P(|S_n|_{\rho_2}/\sigma>t) \to P(|B|_{\rho_2}>t)$, but only for sufficiently large $t>t_0$, and we can provide exact bounds on $t_0$ in terms of the tails of $\eta_t$.
The lower bound on $t$ is not merely a deficiency of our proof, but the weak convergence actually does not hold below this threshold; see Proposition \ref{prop:lower-tail} and Figure \ref{fig:ExceedancePlot}.
We term this novel phenomenon \textit{thresholded weak convergence}, and we study its properties in more detail in Section \ref{sec:weaktail}.
The implication of this probabilistic result for statistical inference is that we may choose critical values based on the pivotal limit distribution of $|B|_{\rho_2}$, and for sufficiently small significance level $\alpha\leq \alpha_0$, these critical values will be asymptotically valid. 
We stress that $\alpha_0$ does not vanish as $n\to \infty$, but is a fixed value.

Our new results on thresholded weak convergence are broadly applicable beyond the prototypical signal discovery problem \eqref{eqn:model-signal}, and beyond iid observations. 
In Section \ref{sec:weaktail}, we specifically showcase how to apply our results to nonstationary time series. 
We also derive a novel sub-Gaussian concentration inequality for nonlinear time series (Theorem \ref{thm:concentration-dependent}), which might be of independent interest.
After discussing the signal discovery problem in detail in Section \ref{sec:stat-signal}, we formulate a general goodness-of-fit test in Section \ref{sec:gof}, and propose novel inferential procedures for multiple changepoint detection in Section \ref{sec:change}. In Section \ref{sec:nonstationary}, we describe how to derive critical values for nonstationary errors via a multiplier bootstrap. 

The multiscale test statistic $T_n^{\text{DS}}$ has been criticized by \citet{walther_calibrating_2022} for loosing too much finite sample power against short signals. 
We address this criticism via variants of test statistics $|S_n|_\rho$ for the signal discovery problem based on alternative moduli $\rho$ such that $\rho(h)/\rho_2(h)\to 1$ as $h\to 0$.
Thus, the procedure maintains its optimal asymptotic properties, while being more sensitive to short signals.
Simulations presented in Section \ref{sec:stat-signal} confirm the improved finite sample detection performance for signals of different lengths.
An extended simulation study is presented in Section \ref{sec:simulations}, where we assess the distributional approximation under the null hypothesis and illustrate our methodology on a simulated data set.
Moreover, we apply our procedures to measurements of the power grid frequency during the April 2025 blackout on the Iberian peninsula, and identify various anomalies preceding the ultimate blackout.
All mathematical proofs are gathered in the Appendix.

\subsection*{Notation}
For two sequences $a_n, b_n$ of real numbers, we denote $a_n\ll b_n$ to mean $a_n/b_n\to 0$ as $n\to\infty$.
We denote weak convergence of a random element by $\wconv$, weak convergence of the finite dimensional marginals of a sequence of stochastic processes by $\wconv_{\mathrm{fidi}}$, and thresholded weak convergence of random variables by $\wconv_\tau$ (introduced in Section \ref{sec:weaktail}). 
The space of continuous functions on $[0,1]$ is denoted by $C[0,1]$.
A modulus of continuity is a continuously increasing function $\rho:[0,1]\to (0,\infty)$ with $\rho(0)=0$, and the corresponding Hölder-type seminorm on $C[0,1]$ is $|f|_\rho = \sup_{u,v\in[0,1]} |f(u)-f(v)|/\rho(|u-v|)$, inducing the Hölder-type Banach space $C^\rho = \{ f\in C[0,1] \,:\, \|f\|_\rho = \|f\|_\infty + |f|_\rho \}$. Of particular relevance are the special cases $\rho_\alpha(h) = \sqrt{h}\log^{\frac{1}{\alpha}}(e/h)$, for $\alpha\in[0,2]$, and $\rho_{2,a}(h)=\sqrt{h (a+\log(e/h))}$. For a given modulus of continuity we denote by $C^\rho_0$ the subspace of functions $f\in C^\rho$ such that
\begin{align*}
    \lim_{\delta\to 0}\sup_{0<|u-v|<\delta}\frac{|f(u)-f(v)|}{\rho(|u-v|)}=0.
\end{align*}
For two moduli of continuity $\rho$ and $\rho$' we denote by $\rho\gg\rho'$ that $\rho'(h)/\rho(h)\to 0$ as $h\to 0$. The interpolated partial sum process based on $Y_t$ is denoted as $S_n(u) = \frac{1}{\sqrt{n}} \sum_{t=1}^{\lfloor un\rfloor} Y_t + \frac{un-\lfloor un\rfloor}{\sqrt{n}} Y_{\lfloor un\rfloor +1}$, and $\widetilde{S}_n(u)$ denotes the interpolated partial sum process based on the noise variables $\eta_t$, that is $\tilde S_n(u) = \frac{1}{\sqrt{n}} \sum_{t=1}^{\lfloor un\rfloor} \eta_t + \frac{un-\lfloor un\rfloor}{\sqrt{n}} \eta_{\lfloor un\rfloor +1}$.

\section{Asymptotic theory}\label{sec:weaktail}

Central to our statistical methodology is a detailed asymptotic treatment of the interpolated partial sum process of the random variables $\eta_t$,
\begin{align*}
    \frac{1}{\sqrt{n}} \sum_{t=1}^{\lfloor un\rfloor} \eta_t + \frac{un-\lfloor un\rfloor}{\sqrt{n}} \eta_{\lfloor un\rfloor +1}.
\end{align*}
The process $\widetilde{S}_n$ is a random element in the Polish space $C[0,1]$, and for iid $\eta_t$,  Donsker's Theorem \citep{Billingsley1999} establishes its weak convergence towards a Brownian motion, $\widetilde{S}_n \wconv \sigma B$.
The continuous mapping theorem implies that $T(\widetilde{S}_n)\wconv T(\sigma B)$ for any continuous functional $T:C[0,1]\to \R$. 
In an inferential setting, $T(S_n)$ is a test statistic, and the continuity requirement restricts the class of possible  statistics. 
Since the paths of $u\mapsto \widetilde{S}_n(u)$ are highly regular as piecewise linear functions, one may instead consider weak convergence in the smaller Hölder-type spaces $C^{\rho}_0$ for modulus of continuity $\rho$, thus allowing for a broader class of functionals; see \cite{lamperti_convergence_1962} for an early treatment of this idea.
The boundary case is marked by $\rho_2(h)=\sqrt{h \log(e/h)}$, which is the modulus of continuity of the limiting Brownian motion such that $B\in C^{\rho_2}$, but not $B\in C^{\rho_2}_0$. 
Thus we can not expect a central limit theorem in $C^{\rho_2}_0$, but at best in $C^{\rho_2}$. 
However, as Proposition \ref{prop:lower-tail} below shows, $\widetilde{S}_n$ does not converge weakly in $C^{\rho_2}$.

As a way to resolve this situation, we observe that not all continuous functionals $T:C^{\rho_2}\to \R$ are relevant for statistical inference. 
As outlined in the introduction, we are particularly interested in the functional $T(\widetilde{S}_n)=|\widetilde{S}_n|_{\rho_2}$, i.e.\ the Hölder-type seminorm in $C^{\rho_2}$.
We show that the distribution of this specific random variable converges on the majority of its support, but not everywhere. 
To formalize this notion, we introduce the new concept of thresholded weak convergence as follows.

\begin{definition}
A a sequence of real-valued random variables $(X_n)_n$ converges weakly beyond a threshold to a random variable $X$, if there exists some $\tau\in\R$ such that 
\begin{align*}
    P(X_n>t)\to P(X>t)\quad\forall t>\tau, \; t\in \mathcal{C}(X),
\end{align*}
where $\mathcal{C}(X)$ denotes the points of continuity of the distribution function of $X$.
We denote this by $X_n\wconv_{\tau} X$.
Equivalently, $(X_n \vee  \tau) \wconv (X\vee \tau)$.
\end{definition}

Our central probabilistic result is the following sufficient criterion for thresholded weak convergence of Hölder-type seminorms.

\begin{theorem}\label{thm:tail-convergence}
    Let $W_n$ be a sequence of stochastic processes with paths in $C[0,1]$, and let $\rho_0:[0,1]\to[0,\infty)$ an increasing function with $\rho_0(0)=0$.
    Suppose there exist $C>0$ and for any $t>C$ a $\kappa(t)>1$ and $K(t)>0$, such that  
    \begin{align}
        P\left(\frac{|W_n(u)-W_n(v)|}{\rho_0(|u-v|)} >t\right) \leq K(t) |u-v|^{\kappa(t)},\qquad u,v\in[0,1]. \label{eqn:tail-tight}\tag{T}
    \end{align}
    Assume further that there exist $p>0$ and $\widetilde{K}>0$ such that 
    \begin{align}
        \rho_0(zh)\leq \widetilde{K} z^p \rho_0(h), \qquad h,z\in[0,1].\label{eqn:rho-fac}\tag{R}
    \end{align}
    Then, for any $t>C$,
    \begin{align}
        \sup_n P\left( \sup_{u,v\in [0,1], |u-v|\leq h} \frac{|W_n(u)-W_n(v)|}{\rho_0(|u-v|)} > t \right) \longrightarrow 0, \quad \text{as } h\to 0, \label{eqn:cont-modulus}
    \end{align}
    and for any modulus of continuity $\rho$ such that $\rho(h)/\rho_0(h)\to \infty$ as $h\to 0$, the sequence $|W_n|_{\rho}$ is stochastically bounded.
    If furthermore there exists a limiting process $W$ such that $W_n\wconv_{\mathrm{fidi}} W$, then 
    \begin{align}
    \begin{split}
        W_n\wconv W \text{ in } C^\rho_0, &\qquad\text{if } \rho(h)/ \rho_0(h)\overset{h\to 0}{\longrightarrow} \infty, \\
        |W_n|_{\rho} \wconv_{C} |W|_{\rho}, &\qquad \text{if }\rho(h)/\rho_0(h)\overset{h\to 0}{\longrightarrow} 1.
    \end{split}\label{eqn:wconv}
    \end{align}
\end{theorem}

The previous result also extends to processes with values in metric spaces, see Theorem \ref{thm:metric} in the Appendix.
In the sequel, we will mostly use Theorem \ref{thm:tail-convergence} for the modulus $\rho_0(h)=\rho_\alpha(h)=\sqrt{h}|\log(e/h)|^{\frac{1}{\alpha}}$, which satisfies \eqref{eqn:rho-fac}, see Lemma \ref{Lemma:SchillVarianteWeibull}.
In this situation, a sufficient condition for \eqref{eqn:tail-tight} is
\begin{align}
    P\left(|W_n(u)-W_n(v)| >r\,C \sqrt{|u-v|}\right) \leq K \exp(-r^\alpha),\quad u,v\in[0,1], \;r>0. \label{eqn:tail-tight-alpha}\tag{T-$\alpha$}
\end{align}
The tightness condition \eqref{eqn:tail-tight-alpha} can be interpreted as an entropy bound in terms of the sub-Weibull Orlicz norms \citep{vladimirova_subweibull_2020}.
For the special case that $W_n$ is a standard Brownian motion, \eqref{eqn:tail-tight-alpha} holds with $C=\sqrt{2}$ and $\alpha=2$, and \eqref{eqn:cont-modulus} matches Levy's result on the modulus of continuity of Brownian motion \cite[Thm.~10.6]{schilling_brownian_2021}.

\begin{remark}
    Invariance principles in Hölder spaces have been first studied by \cite{lamperti_convergence_1962}, for the moduli $\rho(h)=h^\gamma$, who formulates the sufficient condition that for some $a,b>0$ with $\gamma<\frac{b}{a}$, the moment condition
    \begin{align*}
        \E|W_n(u)-W_n(v)|^a \leq C |u-v|^{1+b}
    \end{align*}
    holds. 
    Via Markov's inequality, this can be transformed into the form \eqref{eqn:tail-tight} with suitable modulus of continuity.
    For the statistically most relevant case $W_n=\widetilde{S}_n$, \cite{rackauskas_necessary_2004} reformulates the sufficient conditions in terms of polynomial tail bounds.
    Results for dependent data under mixing conditions are due to \cite{hamadouche_invariance_2000}, and for linear processes in Hilbert spaces by \cite{rackauskas_holderian_2009}.
    The stronger modulus $\rho_\alpha(h) = \sqrt{h} \log^\frac{1}{\alpha}(e/h)$ for $\alpha<2$ is studied by \cite{rackauskas_necessary_2004-1}, who also allow for observations in infinite dimensional Banach spaces.
    For the latter modulus, they formulate the condition $P(\|Y_t\|>\log^\frac{1}{\alpha}(er)) \ll 1/r$ as $r\to \infty$, or equivalently $P(\|Y_t\|>z) \ll \exp(-z^\alpha)$ as $z\to\infty$.
    As all these studies aim for a functional CLT, they necessarily omit the critical case $\rho_2$, where classical weak convergence fails. 
    Our result on thresholded weak convergence shows that in the critical case, one can still obtain certain distributional limits which is sufficient for many statistical purposes, as described in Section \ref{sec:statistics}.
\end{remark}

\begin{remark}\label{rem:discretization}
    Property \eqref{eqn:cont-modulus} also implies that $|W_n|_{\rho}$ and the corresponding discretized variant $\max_{i,j=1,\ldots, n} |W_n(\frac{i}{n}) - W_n(\frac{j}{n})|/\rho(\frac{|i-j|}{n})$ have the same limit distribution as $n\to\infty$, which facilitates numerical implementation.
\end{remark}

As a theoretical complement to the definition of thresholded weak convergence, we provide a version of Skorokhod's Representation Theorem. 
\begin{theorem}[Skorokhod representation]\label{thm:Skorokhod-tail}
    Let $W_n\in C[0,1]$ such that $W_n\wconv_{\mathrm{fidi}} W$, and $|W_n|_{\rho^*}\wconv_{T} |W|_{\rho^*}$.
    Then there exists a probability space with random elements $W_n$ and $W$ such that
    \begin{align}
    | W_n-W|_{\rho} \to 0,\qquad \forall \rho \gg \rho^*. \label{eqn:Skorokhod-1}
    \end{align}
\end{theorem}
We emphasize that Theorem \ref{thm:Skorokhod-tail} yields a single coupling which approximates the limit simultaneously in all Hölder-type spaces $C^{\rho}$ with $\rho\gg \rho^*$.


We proceed to specify the abstract Theorem \ref{thm:tail-convergence} for the statistically interesting partial sum process $\widetilde{S}_n$.
If the random variables $\eta_t$ are sub-Gaussian, we obtain the following extension of Donsker's Theorem.
\begin{corollary}\label{cor:Donsker}
    Let $\eta_t$ be iid centered random variables with unit variance such that $\E \exp(r\eta_t)\leq \exp(\frac{r^2}{C^2})$. 
    Then the interpolated partial sum process $\widetilde{S}_n$ satisfies \eqref{eqn:wconv} with threshold $C$, exponent $\alpha=2$, and the limit process $W$ is a standard Brownian motion.
\end{corollary}

What happens in the lower part of the distribution in \eqref{eqn:wconv}? The following proposition shows that the thresholding is not just an artifact of our proof, but indeed necessary. 
In other words, the lower part of the distribution depends on the distribution of $S_n$ and might converge to various different limits depending on the tails of $W_n$. 
This theoretical finding is also supported by simulation results depicted in Figure \ref{fig:ExceedancePlot}.
\begin{proposition}\label{prop:lower-tail}
    For any $T>0$, there exist sub-Gaussian iid random variables $\eta_t$ with $\Var(\eta_t)=1, \E(\eta_t)=0$, and $\|\eta_t\|_{\psi_2}<\infty$, such that Theorem \ref{thm:tail-convergence} applies with $\rho_0=\rho_2$, but $\liminf_{n\to\infty} P(|\widetilde{S}_n|_{\rho_2}\geq T) = 1$.
\end{proposition}

An essential statistical benefit of the multiplicatively weighted statistic $T_n^*$ is that its limit theory can be readily extended to other sampling situations. 
Here, we study a general nonlinear, non-stationary time series given by the very general Bernoulli shift model $\eta_t = G_t(\beps_t)$ for $\beps_t = (\epsilon_t, \epsilon_{t-1},\ldots)$ with $\epsilon_i \sim U(0,1)$ iid random seeds. 
To quantify the temporal dependence, we introduce $\beps_{t,h}=(\epsilon_t,\ldots, \tilde{\epsilon}_{t-h},\ldots)$ for an independent copy $\tilde{\epsilon}_i\sim U(0,1)$.
Following the general idea of \cite{Wu2005}, we define the physical dependence measure w.r.t.\ the sub-Gaussian norm $\|X\|_{\psi_2} = \inf\{c:\E \exp(X^2/c^2)\leq 2\}$ as
\begin{align*}
    \delta_{\psi_2}(h) = \sup_t \| G_t(\beps_t) - G_t(\beps_{t,h})\|_{\psi_2}.
\end{align*}
With this concept, we obtain the following novel sub-Gaussian concentration bound for dependent data, which might be of independent interest.

\begin{theorem}[Sub-Gaussian concentration with dependence]\label{thm:concentration-dependent}
    There exists a universal $K$ such that for any centered time series $\eta_t$ of the form $\eta_t=G_t(\beps_t)$,
    \begin{align*}
        \left\| \sum_{t=1}^n w_t \eta_t \right\|_{\psi_2} \leq K\sqrt{\sum_{t=1}^n |w_t|^2} \left(\sum_{j=1}^\infty \sqrt{j} \delta_{\psi_2}(j)\right),
    \end{align*}
    where $w_t\in\R$ is a sequence of weights.
\end{theorem}

\begin{remark}
    The physical dependence measure is usually defined in terms of $L_p(P)$ norms instead of sub-Gaussian Orlicz norms, and denoted as $\delta_p(j)$.
    For the latter, it holds \citep{Liu2013}
    \begin{align*}
        \left\| \sum_{t=1}^n \eta_t \right\|_{L_p} \leq C \sqrt{n} \sum_{j=0}^\infty \delta_{p}(j).
    \end{align*}
    This bound differs from Theorem \ref{thm:concentration-dependent} by a factor $\sqrt{j}$ in the series. 
    This difference may be explained as follows: 
    For iid random variables $Z_i$, we have $\|\sum_{i=1}^k Z_i\|_{L_p} \leq  C \sqrt{k}\max_i \|Z\|_{L_p}$, and $\|\sum_{i=1}^k Z_i\|_{\psi_2} \leq C \sqrt{k}\max_i \|Z_i\|_{\psi_2}$. 
    For the $L_p$ norm, this inequality remains valid if the $Z_i$ are martingale differences \citep{Pinelis1994a}.
    On the other hand, extending the sub-Gaussian inequality to martingales requires a uniform bound on the sub-Gaussian norm of $Z_i$ conditional on $\mathcal{F}_{i-1}$, that is
    \begin{equation*}
        \| Z_i\|_{\psi_2,\mathcal{F}_{i-1}}:= \inf\left\{ c>0\,:\, \E\left( \exp(Z_i^2/c)|\mathcal{F}_{i-1}\right) \leq 2 \right\},
    \end{equation*}
    similar to the Azuma--Hoeffding inequality.
    This is more restrictive than imposing an upper bound on $\|Z_i\|_{\psi_2}$.
    In contrast, our novel formulation of the sub-Gaussian physical dependence does not impose this kind of uniformity, and this weaker assumption leads to the extra term $\sqrt{j}$ in the concentration inequalities.
\end{remark}

To ensure that the partial sum process of nonstationary random variables has a well-defined limit, we introduce the additional assumption that the noise process is locally stationary. 
This concept was initially introduced by \cite{Dahlhaus1997}, see also \cite{Dahlhaus2017}, and consists of rescaling the non-stationarity in $t\in\{1,\ldots, n\}$ to relative times $t/n\in[0,1]$.
Specifically, we suppose that $\eta_t = \eta_{t,n} = G_{t,n}(\beps_t)$ forms a triangular array of random variables, and for any $u\in[0,1]$, there exists a measurable $\tilde{G}_u:\R^\infty\to \R$ such that
\begin{align}
    \int_0^1 \left\| G_{\lfloor un\rfloor, n}(\beps_0) - \tilde{G}_u(\beps_0) \right\|_{L_2}\, du \longrightarrow 0. \label{ass:LS1} \tag{LS-1}
\end{align}
Moreover, we impose bounded variation of the mapping $t\mapsto G_{t,n}$, that is,
\begin{align}
     \sup_n \left(\|G_{1,n}(\beps_0)\|_{L_2} + \sum_{t=2}^n \|G_{t,n}(\beps_0)-G_{t-1,n}(\beps_0)\|_{L_2}\right)  \;<\;\infty. \label{ass:LS2} \tag{LS-2}
\end{align}
The formulation of local stationarity in terms of the Bernoulli shift model is due to \cite{Wu2011} under the condition that $G_{t,n}=\tilde{G}_{t/n}$ and $u\mapsto \tilde{G}_u$ is Lipschitz, and \cite{Zhou2013} extended this to piecewise-Lipschitz with finitely many discontinuities. 
In prior work \citep{mies_functional_2023}, we relaxed this regularity requirement to bounded $p$-variations, with \eqref{ass:LS2} being a special case for $p=1$. 
Assumption \eqref{ass:LS1} was introduced in \cite{mies_strong_2024} as a relaxation of the assumption that $G_{\lfloor un\rfloor, n} \to \tilde{G}_u$ uniformly in $u$.
Under the additional assumption 
\begin{align}
    \delta_{\psi_2}(h) = \mathcal{O}(h^{-\beta})\label{ass:LS3}\tag{LS-3}
\end{align}
for some $\beta>1$, we find that $\widetilde{S}_n(u)\wconv B(\Sigma(u))\deq \int_0^u \sigma_\infty(v)\, dB_v$ for a standard Brownian motion $B$ and $\Sigma(u)=\int_0^u \sigma^2_\infty(v)\, dv$, where $\sigma^2_\infty(v) = \sum_{h=-\infty}^\infty \Cov(\tilde{G}_v(\beps_0), \tilde{G}_v(\beps_h))$ is the local long-run variance.
If we combine this with the novel concentration inequality of Theorem \ref{thm:concentration-dependent}, we are able to derive a multiscale central limit theorem.

\begin{corollary}\label{cor:timeseries}
    Suppose that $\eta_t=G_{t,n}(\beps_t)$ satisfies \eqref{ass:LS1}, \eqref{ass:LS2}, and \eqref{ass:LS3} for some $\beta>\frac{3}{2}$. 
    Then the interpolated partial sum process $\widetilde{S}_n$ satisfies \eqref{eqn:wconv} for some threshold $C$, and exponent $\alpha=2$. and the limit process is $W(u)=B(\Sigma(u))$.
\end{corollary}

Besides the thresholded weak convergence for the critical modulus $\rho_2$, Corollary \ref{cor:timeseries} yields a Hölderian invariance principle for dependent data.
Unlike the results of \cite{hamadouche_invariance_2000} and \cite{rackauskas_holderian_2009}, we impose ergodicity in the form of the physical dependence measure, instead of mixing conditions or models in terms of linear processes. 
Moreover, we also allow for nonstationarity.

\section{Statistical applications}\label{sec:statistics}

\subsection{Signal discovery}\label{sec:stat-signal}

Returning to the signal discovery problem described in the introduction, we consider the setting of iid noise terms $\eta_t$ which are sub-Gaussian such that $\E \exp(r\eta_t) \leq \exp(\frac{r^2}{C_\eta^2})$, but not necessarily normally distributed. 
The multiplicatively weighted multiscale test rejects the null hypothesis $\mathbb{H}_0:\mu_n=0$ for large values of $T_n^* = \max_I T_n(I)/\sqrt{\log \frac{en}{|I|}} = |S_n|_{\rho_2}$, which converges weakly beyond a threshold to $\sigma |B|_{\rho_2}$.
The variance $\sigma^2$ may be estimated as $\widehat{\sigma}_n^2 = \frac{1}{n-1} \sum_{t=2}^{n} (Y_t-Y_{t-1})^2/2$.
Compared to the usual sample variance, this estimator is consistent under the alternative as long as $\mu_n\ll \sqrt{n}$.
As a result, we obtain a feasible test which consistently detects a signal at the optimal rate.
To choose a critical value for a significance level $\alpha\in(0,1)$, we denote by $q_\alpha$ the $(1-\alpha)$-quantile of $|B|_{\rho_2}$ for a standard Brownian motion $B$, which is tabulated in the first row of Table \ref{tab:quantiles}. 
Note that $|B|_{\rho_2} \geq \sqrt{2}$ almost surely, such that $q_\alpha>\sqrt{2}$,

\begin{table}[tb]
    \centering
    \footnotesize
    \begin{tabular}{lrrrrr}
    \toprule
         $\alpha$& $10\%$ & $5\%$ & $1\%$ & $0.1\%$ & $0.01\%$  \\ \midrule
         $\rho_{2,0}=\rho_{2}$ & 2.384 & 2.601 & 3.084 & 3.695 & 4.229\\ 
         $\rho_{2,50}$ & 0.641 & 0.661 & 0.704 & 0.758 & 0.808\\
         $\rho_{2,100}$ & 0.468 & 0.484 & 0.516 & 0.559 & 0.599 \\
         $\rho_{2,500}$ & 0.216 & 0.222 & 0.237 & 0.256 & 0.274\\
         $\rho_{2,1000}$ & 0.153 & 0.158 & 0.169 & 0.182 & 0.191 \\ \midrule
        \multicolumn{2}{l}{Sparse grid $\mathcal{G}_{\text{RW}}$} \\ \midrule   
        $\rho_{2,0}=\rho_{2}$ & 2.173 & 2.370 & 2.824 & 3.408 & 3.928\\ 
        $\rho_{2,50}$  &0.636 & 0.656 &  0.700 & 0.753 & 0.807\\
        $\rho_{2,100}$ & 0.466 & 0.481 & 0.513 &  0.556 & 0.599 \\
        $\rho_{2,500}$ & 0.215 &  0.222 & 0.236 & 0.256 & 0.273 \\
        $\rho_{2,1000}$ & 0.153 & 0.158 & 0.168 & 0.181 & 0.193  \\ \midrule
        \multicolumn{2}{l}{Sparse grid $\mathcal{G}_{\text{dyadic}}$} \\ \midrule
        $\rho_{2,0}=\rho_{2}$ & 1.907 & 2.118 & 2.631 & 3.316 & 3.904\\
        $\rho_{2,50}$ & 0.586 & 0.606 & 0.650 & 0.706 & 0.761 \\
        $\rho_{2,100}$ & 0.431 & 0.446 & 0.479 & 0.518 & 0.557 \\
        $\rho_{2,500}$ & 0.199 & 0.206 & 0.221 & 0.239 & 0.254 \\
        $\rho_{2,1000}$ & 0.141&  0.146 & 0.156 & 0.170 & 0.182  \\
        \bottomrule
    \end{tabular}
    \caption{$(1-\alpha)$-quantiles of $|B|_{\rho_{2,a}}$ evaluated on various grids based on $10^5$ simulations, discretizing the Brownian motion via $10^4$ grid points.}
    \label{tab:quantiles}
\end{table}

\begin{theorem}\label{thm:signal}
    Under $\mathbb{H}_0$, for any $\alpha$ small enough such that $ q_\alpha > C_\eta$, the test maintains size $\alpha$ asymptotically: 
    \begin{align*}
        \limsup_{n\to\infty} P(T_n^* > \widehat{\sigma}_n q_\alpha) \leq \alpha
    \end{align*}
    Under the sequence of alternatives $f_n(t) = \mu_n \mathds{1}(t\in I_n)$ with $I_n = [a_n, a_n+l_n)\subset [1,n]$, such that $\mu_n^2l_n \gg \log(\frac{en}{l_n})$ and $\mu_n\ll \sqrt{n}$, the test is consistent:
    \begin{align}
        \lim_{n\to\infty}P(T_n^* > \widehat{\sigma}_n q_\alpha) \to 1.\label{eqn:power-signal}
    \end{align}
\end{theorem}

To compare this procedure with the multiscale methodology of \cite{dumbgen_multiscale_2001}, suppose for simplicity that $\sigma=1$ is known.
Our test rejects the null if $T_n(I)\geq \zeta^{\textsc{mult}}(|I|)=q_\alpha \sqrt{\log ({en}/{|I|})}$, whereas the test based on $T_n^{\text{DS}}$ rejects if $T_n(I)\geq \zeta^{\text{DS}}(|I|)=\sqrt{2} \sqrt{\log (en/|I|)} + c_\alpha$ for a critical value $c_\alpha>0$. 
Since $q_\alpha>\sqrt{2}$, the second threshold is sharper for short intervals $|I|$. 
However, it is only valid for Gaussian errors, while our new threshold rule is robust to non-Gaussian errors. 
To further illustrate this difference in robustness, we try to construct a threshold rule of the form $\zeta^{\text{DS}}_A(|I|)=A \sqrt{\log (en/|I|)} + c_\alpha$ for some fixed $A>\sqrt{2}$. 
The following result shows that this is theoretically possible but statistically infeasible, as the value of $A$ will depend on the unknown sub-Gaussian tail bound.

\begin{proposition}\label{prop:DS-A}
    For any $A> C_\eta$, 
    \begin{align*}
        \sup_I \left\{T_n(I) - A \sqrt{\log \tfrac{en}{|I|}}\right\}_+ \;\wconv\; \sup_{u\leq v}\left\{ \tfrac{|B(v)-B(u)|}{\sqrt{v-u}} - A \sqrt{\log \tfrac{e}{|u-v|} } \right\}_+.
    \end{align*}
    On the other hand, for any $A\geq 0$, there exist iid sub-Gaussian random variables $\eta_t$ such that $\sup_I \left\{T_n(I) - A \sqrt{\log \frac{en}{|I|}}\right\}_+ \to \infty$ in probability.
\end{proposition}

While both threshold rules, $\zeta^{\text{DS}}_A$ and $\zeta^{\textsc{mult}}$, are only consistent in the regime $\mu_n^2 l_n \gg \log(n/l_n)$, the former puts slightly more weight on shorter signals. 
Thus, we expect a different assignment of statistical power to the length scales.
To quantify the power against signals of length $l_n$ in finite samples, \cite{walther_calibrating_2022} suggested the \textit{realised exponent} $e_n(l_n)$ defined as
\begin{align*}
    e_n(l_n) = \frac{\mu_{\min}(l_n)}{\sqrt{2}\rho_2(l_n)},
\end{align*}
where $\mu_{\min}(l_n)\geq 0$ is the smallest value such that the test with significance level $\alpha=10\%$ has $80\%$ power against the alternative $f_n(t) = \mu_n\mathds{1}(t\in I_n)$ for a randomly placed interval $I_n\subset[1,n]$ of length $l_n$.
Table \ref{tab:realised-exponent} presents the realised exponents for different threshold rules and Gaussian errors for sample size $n=10^4$. 
SCAN refers to a constant threshold $\zeta^{\textsc{scan}}(|I|)=d_{\alpha,n}$ as a quantile of the uncorrected scan statistic $\sup_{I} T_n(I)$, which necessarily depends on $n$. 
The comparsion reveals that the threshold $\zeta^{\textsc{mult}}$ corresponding to the modulus $\rho_2=\rho_{2,0}$ has less power on shorter length scales. 
To increase the power of the multiplicatively weighted procedure, we consider the alternative modulus of continuity $\rho_{2,a}(h) = \sqrt{h (a+\log(e /h))}$ for $a\geq 0$. 
Since $\rho_2(h)/\rho_{2,a}(h)\to 1$ as $h\to 0$, the theory of Section \ref{sec:weaktail} as well as Theorem \ref{thm:signal} still apply, and Table \ref{tab:quantiles} presents the corresponding critical values. 
Table \ref{tab:realised-exponent} shows that the use of $\rho_{2,a}$ improves the finite sample performance against short signals, being competitive with the DS threshold. In particular, the use of larger $a$ increases the performance of the test based on $|S_n|_{\rho_{2,a}}$ for shorter signals further, though loosing performance at longer signals, converging for increasing $a$ to the performance of SCAN.
In contrast to the benchmark thresholds, the central benefit of our proposed threshold scheme is its validity for non-Gaussian errors thanks to the new asymptotic theory.

\begin{table}[tb]
\centering
\footnotesize
    \begin{tabular}{c|cccccccc}
        \toprule
         $|I_n|$&1&5&10&15&50&100&500&1000  \\ 
         \midrule
         \multicolumn{9}{c}{\textbf{Full grid }$\mathcal{G}$} \\ \midrule
        SCAN & 1.52 & 1.75 & 1.80 & 1.89 & 2.08 &2.25 & 3.08 & 3.74 \\
        DS & 2.21& 2.32& 2.27 &2.22& 2.24& 2.20 &2.22 &2.38\\
         $\rho_{2,0}$& 3.42 &3.42 & 3.38 & 3.27 & 3.15 & 3.04 & 2.81 & 2.58\\
         $\rho_{2,50}$& 1.62&1.74&1.79&1.89&2.03&2.19&2.89&3.46\\
         $\rho_{2,100}$& 1.58&1.76&1.76&1.90&2.06&2.25&3.03&3.56\\
         $\rho_{2,500}$&1.54 & 1.69 & 1.77 & 1.88 & 2.05 & 2.28 &3.13 & 3.69 \\ 
         $\rho_{2,1000}$&1.54 & 1.72 &1.86 & 1.87 & 2.07 & 2.31 & 3.16 & 3.72 \\ 
         $\rho_{2,10^6}$&1.54 & 1.73 &1.79 & 1.84 & 2.12 & 2.29 & 3.09 & 3.74 \\ \midrule
        \multicolumn{9}{c}{\textbf{Sparse grid} $\mathcal{G}_{\text{RW}}$} \\ \midrule 
        SCAN &1.53 & 1.71 & 1.80 & 1.86 & 2.14 & 2.41 & 3.35 & 4.14\\
        DS &2.10& 2.12 &2.10 &2.08& 2.10& 2.14& 2.29 &2.38\\
        $\rho_{2,0}$ & 2.90 & 2.86 & 2.81 &2.78 &2.70 &2.67 &2.44 &2.41\\
        $\rho_{2,50}$& 1.57 & 1.73 &1.79 & 1.83& 2.08& 2.31& 3.14 &3.79\\
        $\rho_{2,100}$& 1.57 & 1.73& 1.77& 1.82 &2.12& 2.34 &3.24 &3.96\\
        $\rho_{2,500}$& 1.55 & 1.71& 1.78 &1.85 &2.15& 2.39 &3.36
        & 4.05\\
        $\rho_{2,1000}$& 1.54 &1.72& 1.80 &1.85& 2.15 &2.43&3.37& 4.13\\
        \midrule
            \multicolumn{9}{c}{\textbf{Dyadic grid} $\mathcal{G}_{\text{dyadic}}$} \\ \midrule 
        SCAN & 2.81 & 1.88& 2.05 &2.21& 2.57& 2.94& 4.15 &4.99 \\
        DS & 3.55 &2.27& 2.34 &2.38 &2.45 &2.46& 2.61 &2.66 \\
        $\rho_{2,0}$ & 4.57& 2.90& 2.93& 2.96 &2.90 &2.90& 2.75& 2.69\\
        $\rho_{2,50}$& 2.85 &1.88 &2.03 &2.19& 2.52 &2.78 &3.82& 4.55\\
        $\rho_{2,100}$& 2.84 &1.89 &2.04 &2.21& 2.53 &2.83& 3.96 &4.76\\
        $\rho_{2,500}$& 2.80 &1.90& 2.06& 2.21 &2.55 &2.90 &4.09 &4.96\\
        $\rho_{2,1000}$& 2.79 &1.90 &2.05& 2.20 &2.59& 2.89& 4.13& 4.99\\
         \bottomrule
    \end{tabular}
    \caption{Realized exponents for Gaussian noise and different threshold rules. We determine the realized exponents based on $2000$ simulations for the full grid, and based on $10^4$ simulations for the sparse grid $\mathcal{G}_{\text{RW}}$ and the dyadic grid $\mathcal{G}_{\text{dyadic}}$.}
    \label{tab:realised-exponent}
\end{table}


\subsection{Goodness-of-fit testing}\label{sec:gof}

The signal discovery problem can be interpreted as testing for a very specific mean function $f_n(t)=0$ in \eqref{eqn:model-signal}. 
More generally, we may perform a goodness-of-fit test for a class of functions $\Fhyp_0\subset\{ f:\{1,\ldots, n\} \to \R \}$.
A simple example consists of parametric classes $\{ f(t)=a+bt\,|\,a,b\in\R \}$ or $\{ f(t) = a \sin(rt)\,|\, a\in\R, r>0 \}$.
Alternatively, we may also test certain shape constraints, such as monotonicity $\Fhyp_{\uparrow}=\{ f\,:\, f(t)\leq f(t+1) \}$, convexity $\Fhyp_{\text{conv}}=\{f(t+1)-f(t) \geq f(t)-f(t-1)\}$, or non-negativity $\Fhyp_{\geq 0}=\{f\geq 0\}$, which are also studied by \cite{dumbgen_multiscale_2001}.
Other shape constraints include unimodality \citep{chatterjee_adaptive_2019}, or S-shape \citep{feng_nonparametric_2022}. 
The choice $\Fhyp_0=\{0\}$ recovers the signal discovery problem, and the case of a singleton null $\Fhyp_0=\{f_0\}$ has been treated by \cite{rohde_adaptive_2008} via a multiscale methodology for symmetric errors.
In this Section, we propose a generic testing procedure which is based on our novel probabilistic results, and which can be used to perform goodness-of-fit tests for all classes discussed above.

As multiscale test statistic for the composite null hypothesis $\mathbb{H}_0: f_n\in \Fhyp_0$, we suggest
\begin{align*}
    T_n^*(\Fhyp_0) &= \inf_{f\in \Fhyp_0} T_n^*(f), \quad \text{where } T_n^*(f) = |S_n^f|_{\rho_2} \\
    \text{and} \quad S_n^f(u) &= \frac{1}{\sqrt{n}} \sum_{t=1}^{\lfloor un\rfloor} [Y_t-f(t)] + \frac{un-\lfloor un\rfloor}{\sqrt{n}} [Y_{\lfloor un\rfloor +1}-f(\lfloor un\rfloor +1)].
\end{align*}
Under the null, we have $T_n^*(\Fhyp_0) \leq T_n^*(f_n) = |\widetilde{S}_n|_{\rho_2}$, which converges weakly beyond a threshold by virtue of the results of Section \ref{sec:weaktail}. 
Thus, we may combine the critical values reported in Table \ref{tab:quantiles} with an estimate of the variance to construct an asymptotically valid multiscale test procedure for the goodness-of-fit problem.

\begin{theorem}\label{thm:gof}
    Suppose $\eta_t$ are iid sub-Gaussian such that $\E \exp(r\eta_t)\leq \exp(\frac{r^2}{C_\eta^2})$, and $\sum_{t=2}^n |f_n(t)-f_n(t-1)|^2 \ll n$.
    Under $\mathbb{H}_0$, for any $\alpha$ small enough such that $q_\alpha>C_\eta$,
    \begin{align*}
        \limsup_{n\to\infty} P\left( T_n^*(\Fhyp_0) > \widehat{\sigma}_n q_\alpha \right) \leq \alpha,
    \end{align*}
    for $q_\alpha$ and $\widehat{\sigma}_n^2$ as in Section \ref{sec:stat-signal}.
\end{theorem}

In addition to a global statement about the structural properties of $f_n$, one is often also interested in its local shape.
That is, we want to infer if $f_n$ is convex or monotone or non-negative on a given interval $I$.
The goodness-of-fit test may also be localized to obtain such insight into the qualitative nature of the regression function. 
To this end, for any interval $I\subset[1,n]$ and a candidate class $\Fhyp_0$, we define the localization
\begin{align*}
    \Fhyp_0^I = \{ f:\{1,\ldots, n\} \to \R \,|\, \exists \tilde{f}\in\Fhyp_0 \text{ such that } \tilde{f}(t)=f(t) \text{ for all } t\in I \} \supset \Fhyp_0.
\end{align*}
Since the localization is a weaker requirement, we have $T_n^*(\Fhyp_0^I) \leq T_n^*(\Fhyp_0)$ and $\sup_{I} T_n^*(\Fhyp_0^I) = T_n^*(\Fhyp_0)$.
This means that the critical value $\widehat{\sigma}_n q_\alpha$ allows for testing all local hypotheses simultaneously while controlling the type I error. 
Inverting these tests yields the class of intervals
\begin{align*}
    \mathcal{I}_n = \left\{ I \,|\, T_n^*(\Fhyp_0^I) > \widehat{\sigma}_n \cdot q_\alpha \right\}.
\end{align*}
In line with \cite{fryzlewicz_narrowest_2024}, we call any $I\in \mathcal{I}_n$ an interval of significance. 
For example, if $\Fhyp_0$ is the class of increasing functions, then $I\in \mathcal{I}_n$ is interpreted as evidence that the function $f_n$ has a strict local minimum or maximum in $I$. 

\begin{proposition}\label{prop:localized}
    Under the conditions of Theorem \ref{thm:gof}, with asymptotic probability at least $1-\alpha$, all intervals of significance are true discoveries, that is
    \begin{align*}
        \limsup_{n\to\infty} P\left( f_n\notin \Fhyp_0^I \text{ for all } I\in \mathcal{I}_n \right) \geq 1-\alpha.
    \end{align*}
\end{proposition}

The asymptotic power of the goodness-of-fit test can be analyzed for special cases upon identifying alternatives of interest. 
In the next section, we pursue this for the commonly studied problem of multiple changepoint testing, and show that our procedure achieves asymptotically optimal detection performance.

\subsection{Multiple changepoint detection}\label{sec:change}

In the multiple changepoint problem \citep{cho_data_2021}, the regression function $f_n$ is of the form
\begin{align}
    f_n(t) = \delta_0 + \sum_{k=1}^\kappa \delta_k \mathds{1}(t\geq \tau_k). \label{eqn:model-change}\tag{CP}
\end{align}
In words, $f_n$ is a step function with $\kappa$ jumps at locations $\tau_k\in \{2,\ldots, n\}$ and jump sizes $\delta_k\in\R$.
This model is common in applications as the changepoints $\tau_k$ can clearly be interpreted as points of interest, e.g.\ to identify regions with interesting copy number variations in DNA \citep{niu_screening_2012}, or changes in measurement techniques or observation locations in climate time series \citep{reeves_review_2007}.
The main statistical objective is to perform inference on (a) the number $\kappa$ of changes, and (b) the locations $\tau_k$.

For changepoint inference, Hölder-type statistics similar to our multiplicative thresholds have already been suggested by \citet{rackauskas_holder_2004}, \citet{rackauskas_estimating_2007}, and \citet{rackauskas_convergence_2020}, for the special case of an epidemic change, i.e.\ $\kappa=2$, although with a sub-optimal modulus of continuity.
An additive multiscale procedure for the multiple changepoint problem has been proposed by \cite{frick2014multiscale}. 
For any $J\in \N$, they formulate the problem as a goodness-of-fit test for $\Fhyp_0^J =\{ f \text{ of the form }\eqref{eqn:model-change} \text{ with $\kappa\leq J$} \}$, and estimate $\kappa$ by the smallest number $J$ such that the hypothesis $f_n\in\Fhyp_0^J$ is accepted. 
Confidence statements for $\tau_k$ may then be derived by test inversion, including all functions $f\in \Fhyp_0^J$ such that the null $f_n=f$ is not rejected. 
They use a test statistic of the form \eqref{eqn:DS}, and impose a lower bound of order $\log(n)^3$ on the interval length, leading to suboptimal performance for short-lived changes; see Proposition \ref{prop:powerless}.

Here, we derive confidence statements  following the alternative approach of narrowest significance pursuit, inspired by \cite{fryzlewicz_narrowest_2024}.
The central idea is to introduce the class $\Fhyp_{\text{const}} = \{ f_n \text{ constant} \}$ of constant regression functions, observing that $f$ has no changepoint in the interval $I$ if and only if $f\in \Fhyp_{\text{const}}^I$. 
As a special case of the goodness-of-fit test of Section \ref{sec:gof}, the localized changepoint test statistic takes the form
\begin{align*}
    T_n^*\left(\Fhyp_{\text{const}}^I\right) = 
    \inf_{a\in\R} \sup_{[un,vn]\subset I} \frac{\left|S_n(v)-S_n(u) - [v-u]a\right|}{\rho_2(v-u)}.
\end{align*}
As the right hand side is a convex function of $a$, the minimum is attained and may be determined numerically, e.g.\ via bisection.
We may then obtain intervals of significance for the change locations $\tau_k$ as 
\begin{align*}
    \mathcal{I}_n^\tau = \left\{ I\,|\, T_n^*\left(\Fhyp_{\text{const}}^I\right) > \widehat{\sigma}_n \cdot q_\alpha \right\},
\end{align*}
where $q_\alpha$ and $\widehat{\sigma}_n$ are as in Section \ref{sec:stat-signal}.
Note that although we borrow the terminology of \citet{fryzlewicz_narrowest_2024}, a similar idea of constructing confidence intervals is also described in \citet[Sec.~5.4]{verzelen_optimal_2023}. 
Subsequent works building on this idea include \citet{pilliat_optimal_2023}, \citet{fryzlewicz_robust_2024}, and \citet{gavioli-akilagun_fast_2025}.

Conditions for consistency of the procedure may be formulated in terms of $\tau_k$ and $\delta_k$, and our test turns out to achieve the optimal detection rates for this problem.
In the sequel, all parameters $\tau_k, \delta_k$, and $\kappa$ may depend on $n$ implicitly, and we denote by $\mathcal{D}=\{\tau_1,\ldots,\tau_\kappa\}\subset\{1,\ldots,n\}$ the set of changepoints. 
Moreover, we introduce the notation $L_k = \min(\tau_k-\tau_{k-1}, \tau_{k+1}-\tau_0)$ for the length of the $k$-th changepoint, where $\tau_0=0$ and $\tau_{\kappa+1}=n+1$.
Prior research into statistical lower bounds  \citep{arias-castro_detection_2011,chan_detection_2013,verzelen_optimal_2023} has revealed that the changepoint $\tau_k$ is only detectable if $\Delta_k^2 L_k \gg \log (n/L_k)$.
Thus, for any $z>0$, we define the set $\mathcal{D}(z)=\{ \tau_j\,|\, \delta_j^2 L_j \geq z \cdot \log(\frac{n}{L_j}) \} \subset \mathcal{D} $ of detectable changepoints, and for any $\tau\in \mathcal{D}(z)$, denote by $w(\tau, z) = w(\tau_k,z)= \inf\{ r\;|\; 2\delta_k^2 r \geq z \cdot \log(\frac{n}{2r}) \} \leq \frac{L_k}{2}$ its detectable locality.

\begin{theorem}\label{thm:changepoint}
    Under the conditions of Theorem \ref{thm:gof}, with asymptotic probability at least $1-\alpha$, all intervals of significance contain at least one changepoint, that is
    \begin{align*}
        \limsup_{n\to\infty} P\left( I \cap \mathcal{D} \neq \emptyset  \text{ for all } I\in \mathcal{I}_n^\tau\right) \geq 1-\alpha.
    \end{align*}
    Moreover, detectable changepoints are asymptotically isolated and located:
    \begin{align}
        \lim_{z\to \infty}\limsup_{n\to\infty}\;P\left(
        \begin{array}{c}
             \forall \tau\in \mathcal{D}(z)\; \exists I\in \mathcal{I}_n^\tau\text{ such that } \\ 
             I \cap \mathcal{D} = \{\tau\} \text{ and } |I| \leq  w(\tau,z)
        \end{array} \right) = 1. \label{eqn:power-change}
    \end{align}
\end{theorem}

In comparison, \cite[Thm.~4.1]{fryzlewicz_narrowest_2024} detects the changepoint $\tau_k$ if $\delta_k^2L_k \gg \log(n)$ and localizes it by an interval with length of order $\log(n)/\delta_k^2$. 
For longer intervals resp.\ smaller changes, this localization and detection rate is suboptimal as \cite{fryzlewicz_narrowest_2024} uses a uniform threshold instead of a multiscale correction.
Incorporating the multiscale idea of \cite{dumbgen_multiscale_2001}, \cite{pilliat_optimal_2023} derive the same order $w(\tau_k,z)$ for the length of the interval, but for a different threshold rule which requires knowledge of the sub-Gaussian norm of the errors.
Theorem \ref{thm:changepoint} shows that our procedure attains the same localization rate, while being statistically feasible since we can specify critical values without knowing the exact sub-Gaussian bound. 
Moreover, our asymptotic treatment readily allows for other sampling settings, in particular non-stationary and temporally dependent data as detailed in the next subsection.

While the detectability condition for a changepoint matches the optimal lower bound of \cite{verzelen_optimal_2023}, the latter authors show that for Gaussian errors the localization rate can be improved to $\mathcal{O}(1/\delta_k^2)$. 
It is not clear if this sharper localization is also attainable for non-Gaussian errors.

The class $\mathcal{I}_n^\tau$ of significant intervals is highly redundant, and contains many overly large as well as intersecting intervals. 
To interpret the statistical findings, it is preferable to report disjoint intervals, specifically many small intervals, for high statistical power.
This can be achieved by a postprocessing of $\mathcal{I}_n^\tau$. 
For any class of intervals $\mathcal{I}$, define $\NSP(\mathcal{I})$ as the output of the following \textit{narrowest significance pursuit} routine:
\begin{enumerate}
    \item Initialize $\mathcal{I}^0=\mathcal{I}$.
    \item Choose $\widehat{I}_k$ as the shortest interval in $\mathcal{I}^{k-1}$, with arbitrary tie breaking.
    \item Update $\mathcal{I}^k= \{  I \in \mathcal{I}^{k-1}\,|\, I\cap \widehat{I}_k = \emptyset\}$.
    \item Terminate after iteration $K$ if $\mathcal{I}^K=\emptyset$, and return $\NSP(\mathcal{I}) = \{  \widehat{I}_1,\ldots, \widehat{I}_K\}$.
\end{enumerate}
This generic formulation is identical to Algorithm 1 of \cite{pilliat_optimal_2023}.

\begin{proposition}\label{prop:NSP}
    Suppose that a class of intervals $\mathcal{I}$, a subset $\widetilde{\mathcal{D}} \subset \mathcal{D}=\{\tau_1,\ldots, \tau_\kappa\}$, and non-negative real numbers $r(\tau)=r(\tau_k) \leq \frac{L_k}{2}$ satisfy
    \begin{itemize}
        \item[i.] $I \cap \mathcal{D} \neq \emptyset$ for all $I\in\mathcal{I}$, and
        \item[ii.] for any $\tau\in\widetilde{\mathcal{D}}$ there exist $I(\tau)\in\mathcal{I}$ such that $I(\tau)\cap \mathcal{D}=\{\tau\}$, and $|I(\tau)| \leq r(\tau)$.
    \end{itemize}
    Then the class $\NSP(\mathcal{I})=\{\widehat{I}_1,\ldots, \widehat{I}_{|\widetilde{\mathcal{D}}|}\}$ consists of exactly $|\widetilde{\mathcal{D}}|$ disjoint intervals, uniquely localizing all $\tau\in\widetilde{\mathcal{D}}$ in the sense of (ii).
\end{proposition} 
That is, the reduced class $\NSP(\mathcal{I})$ satisfies the same statistical guarantees as the full class $\mathcal{I}$, while consisting of disjoint intervals.
In particular, $|\NSP(\mathcal{I}_n^\tau)|$ provides a lower confidence bound on the total number of changepoints.

\subsection{Critical values for nonstationary dependent errors}\label{sec:nonstationary}

The assumption of iid errors is too simplistic for many applications.
Instead, data is often heteroskedastic and temporally dependent, which may be modeled in terms of nonstationary time series.
Especially for changepoint inference, the need to account for nonstationarity was initially observed by \cite{Zhou2013}, who showed how to adapt the critical values of a standard CUSUM statistic. 
Subsequent works include \cite{Vogt2015}, \cite{Dette2018}, \cite{cui_estimation_2021}, and \cite{mies_functional_2023}.
None of these references consider a multiscale threshold, and hence do not achieve simultaneously optimal detection against changes of different lengths.
A multiscale procedure for heteroskedastic, independent Gaussian noise was suggested by \cite{pein_heterogeneous_2017}, however under the restriction that variance and mean change at the same time.
In the sequel, we show how our new asymptotic framework allows for extension to nonstationary time series, i.e.\ to dependent and heteroskedastic data.

If we postulate the model framework $\eta_t=\eta_{t,n}=G_{t,n}(\beps_t)$ introduced in Section \ref{sec:weaktail}, the distributional limit of the partial sum process $\widetilde{S}_n$ is an inhomogeneous Brownian motion $W(u)=B(\Sigma(u))$, and the limit of the goodness-of-fit statistic $T_n^*(f_n)$ is given by the random variable $|W|_{\rho_2}$, see Corollary \ref{cor:timeseries}.
As a consequence of the nonstationarity, determining critical values requires an estimate for the whole function $\Sigma(u)$, instead of just a single value.
To this end, we adapt the estimator of \cite{dette_multiscale_2020}, which is in turn inspired by \cite{wu_inference_2007}, to the nonstationary case.
Specifically, choose a window size $b_n$ such that $1\ll b_n \ll n$ and define
\begin{align*}
    \widehat{\Sigma}_n(u) = \frac{1}{2n}\sum_{t=b_n}^{ un -b_n} \left( \frac{1}{\sqrt{b_n}}\sum_{i=0}^{b_n-1} (Y_{t-i} - Y_{t+1+i}) \right)^2,
\end{align*}
for $u=\frac{i}{n}$, and interpolate linearly in between.
As for the iid case, differencing removes the signal $f_n$, while the block sum is introduced to capture the serial correlation, and partial summation mimics the functional form of $u\mapsto \Sigma(u)$. 
The quadratic terms $\widehat{\sigma}^2_n(\frac{t}{n})=( 1/\sqrt{b_n}\sum_{i=0}^{b_n-1} (Y_{t-i} - Y_{t+1+i}) )^2$ should be interpreted as a noisy estimate of the local long-run variance $\sigma_\infty^2(t/n)$.
Since $\widehat{\Sigma}_n(u)\to \Sigma(u)$ uniformly (Lemma \ref{lem:Sigmahat-LS}), we have $B(\widehat{\Sigma}_n(u))\to B(\Sigma(u))$ uniformly.
Thus, it would be natural to estimate the $(1-\alpha)$-quantile $q_\alpha$ of $|B(\Sigma(u))|_{\rho_2}$ by the corresponding quantile of $|B(\widehat{\Sigma}_n(u))|_{\rho_2}$. 
However, this approach does not yield consistent critical values, as the latter random variable diverges. 
In particular, Levy's Theorem on the modulus of continuity of the Brownian motion shows that $|B(\widehat{\Sigma}_n(u))|_{\rho_2} \geq \sqrt{2} \max_t \widehat{\sigma}_n^2(\frac{t}{n})$, which is stochastically unbounded.

To solve this issue and obtain asymptotically valid critical values, we suggest to use the $(1-\alpha)$-quantile $\widehat{q}_{\alpha,n}$, conditionally on $\widehat{\Sigma}_n$, of the random variable
\begin{align*}
    \sup_{|u-v|>c_n} \frac{|B(\widehat{\Sigma}_n(u)) - B(\widehat{\Sigma}_n(v))|}{\rho_2(|u-v|)},
\end{align*}
for a sequence $c_n$ tending to zero slowly. 
The lower bound on the interval length does lead to correct critical values because under the null, the very short intervals have a vanishing contribution to the distribution of the statistic, as a consequence of \eqref{eqn:cont-modulus}.
On the other hand, under the alternative, these short intervals might carry relevant information, and hence should be included when computing the test statistic.

\begin{theorem}\label{thm:changepoint-LS}
    Let $\eta_t=\eta_{t,n}$ be an array of locally stationary time series satisfying \eqref{ass:LS1}, \eqref{ass:LS2}, and \eqref{ass:LS3} for some $\beta>2$.
    Suppose moreover that $1\ll b_n\ll n$, and
    $\sum_{t=2}^n |f_n(t)-f_n(t-1)|^2 \leq v_n$ for a sequence $v_n\geq 1$, such that
    \begin{align*}
        1 \gg c_n \gg \sqrt{ \frac{v_n b_n}{n} }. 
    \end{align*}
    Then there exists a $\tau>0$, such that for all $\alpha$ small enough such that $q_\alpha>\tau$, we have $\widehat{q}_{\alpha,n} \to q_\alpha$ in probability.
\end{theorem}

In particular, we may use $\widehat{q}_{\alpha,n}$ as critical value in the procedures of Sections \ref{sec:stat-signal}, \ref{sec:gof}, and \ref{sec:change}, and maintain the same statistical guarantees.

\subsection{Sparse grids for faster evaluation}

In practice, the statistic $|S_n|_{\rho_2}$ is evaluated at the grid points $\frac{i}{n}$, leading to a computational cost of $\mathcal{O}(n^2)$. 
A reduction to $\mathcal{O}(n)$ is possible if we restrict attention to a sparse subset of candidate intervals. 
To this end, let $\mathcal{G}\subset [0,1]^2$, and consider the sparse multiscale statistic
\begin{align*}
    |S_n|_{\rho,\mathcal{G}} = \sup_{(u,v)\in \mathcal{G}} \frac{|S_n(u)-S_n(v)|}{\rho(|u-v|)}.
\end{align*}
The same arguments as in Theorem \ref{thm:tail-convergence}, in particular property \eqref{eqn:cont-modulus}, show that $|\widetilde{S}_n|_{\rho_2,\mathcal{G}} \wconv_\tau |W|_{\rho_2,\mathcal{G}}$ beyond a threshold $\tau$. 
In particular, under the conditions of Corollary \ref{cor:Donsker} or Corollary \ref{cor:timeseries}, the limit process is $W$ is Gaussian, and the the threshold $\tau$ is the same as for the full grid.
We may thus either use critical values based on the limit distribution $|W|_{\rho_2,\mathcal{G}}$, or the more conservative critical values based on quantiles of $|W|_{\rho_2}\geq |W|_{\rho_2,\mathcal{G}}$ which are tabulated in Table \ref{tab:quantiles} for iid noise.

We highlight two specific choices of sparse grids.
The first option is the dyadic grid given by
\begin{align*}
    \mathcal{G}_{\text{dyadic}} = \bigcup_{m=1}^\infty \mathcal{G}_{\text{dyadic}, m}, \qquad \mathcal{G}_{\text{dyadic}, m} =\bigcup_{l=0}^{\lfloor\log_2 m\rfloor} \{  (k 2^{-l}, (k+2)\,2^{-l})\,|\, k=0,1,\ldots, 2^l-2  \}.
\end{align*}
The second option, suggested by \citet{rivera_optimal_2013}, is to consider a finer resolution for shorter intervals, specifically $\mathcal{G}_{\text{RW}} = \bigcup_{m=1}^\infty \mathcal{G}_{\text{RW}, m}$ for
\begin{align*}
    \mathcal{G}_{\text{RW}, m} &=\bigcup_{l=0}^{\lfloor\log_2 m\rfloor} \left\{  \left(\tfrac{k}{6\sqrt{l}} 2^{-l}, \tfrac{j}{6\sqrt{l}} 2^{-l} \right)\,\Big|\, k,j=0,\ldots, \lfloor 2^l \cdot 6\sqrt{l}\rfloor \text{ such that } 1\leq \tfrac{|k-j|}{6\sqrt{l}}\leq 2  \right\}.
\end{align*}
In practice, based on sample size $n$, we evaluate the statistic on the grid $\mathcal{G}_{\text{dyadic},n}$ resp.\ $\mathcal{G}_{\text{RW}, n}$. 
Since both sets have a cardinality of order $\mathcal{O}(n)$ resp.\ $\mathcal{O}(n\log n)$, this leads to a significant computational speedup.

All guarantees on the false discoveries remain valid, including Theorem \ref{thm:gof}, Proposition \ref{prop:localized}, the first claim of Theorem \ref{thm:changepoint}, and Theorem \ref{thm:changepoint-LS}.
Moreover, the procedure based on the sparse grid achieves the same asymptotic detection performance as the full grid, as $\mathcal{G}_{\text{dyadic}}$, and thus also $\mathcal{G}_{\text{RW}}\supset \mathcal{G}_{\text{dyadic}}$, satisfy condition \eqref{eqn:condition-grid} below, for $K=3$.

\begin{proposition}\label{prop:sparse}
    Suppose that $\mathcal{G}$ is such that for all $0\leq u < v \leq 1$, there exist $(u',v')\in \mathcal{G}$ such that for some $K\geq 1$,
    \begin{align}
        u'\leq \;u< v\; \leq v' \quad \text{and}\quad |u'-v'| \leq K|u-v|.\label{eqn:condition-grid}
    \end{align}
    Then the sparse test statistic $|S_n|_{\rho_2,\mathcal{G}}$ maintains detection power in the signal discovery problem, i.e.\ \eqref{eqn:power-signal} remains valid.
    Moreover, the sparse procedure isolates and localizes detectable changes in the sense of \eqref{eqn:power-change}.
\end{proposition} 

Table \ref{tab:quantiles} presents the quantiles of $|B|_{\rho_2,\mathcal{G}}$ for both sparse grids, where $B$ is a standard Brownian motion. 
These quantiles serve as critical values for the multiplicatively weighted multiscale tests.
Moreover, the realized exponents for the signal discovery problem with iid Gaussian noise and sample size $n=10^4$ are presented in Table \ref{tab:realised-exponent}. 
It is found that the sparse grid sligthly improves the detection performance for short intervals, and slightly diminishes the power for longer intervals. 
However, the realized exponents are very similar in general, and hence the computational benefit does not incur a major tradeoff in statistical accuracy.

\section{Finite sample accuracy}\label{sec:simulations}

To analyze the finite sample performance of our procedure beyond the realized exponents presented in Table \ref{tab:realised-exponent}, we investigate its performance for non-Gaussian error terms.

\subsection{Size accuracy}

\begin{figure}[tb]
    \centering
    \includegraphics[width=0.49\textwidth]{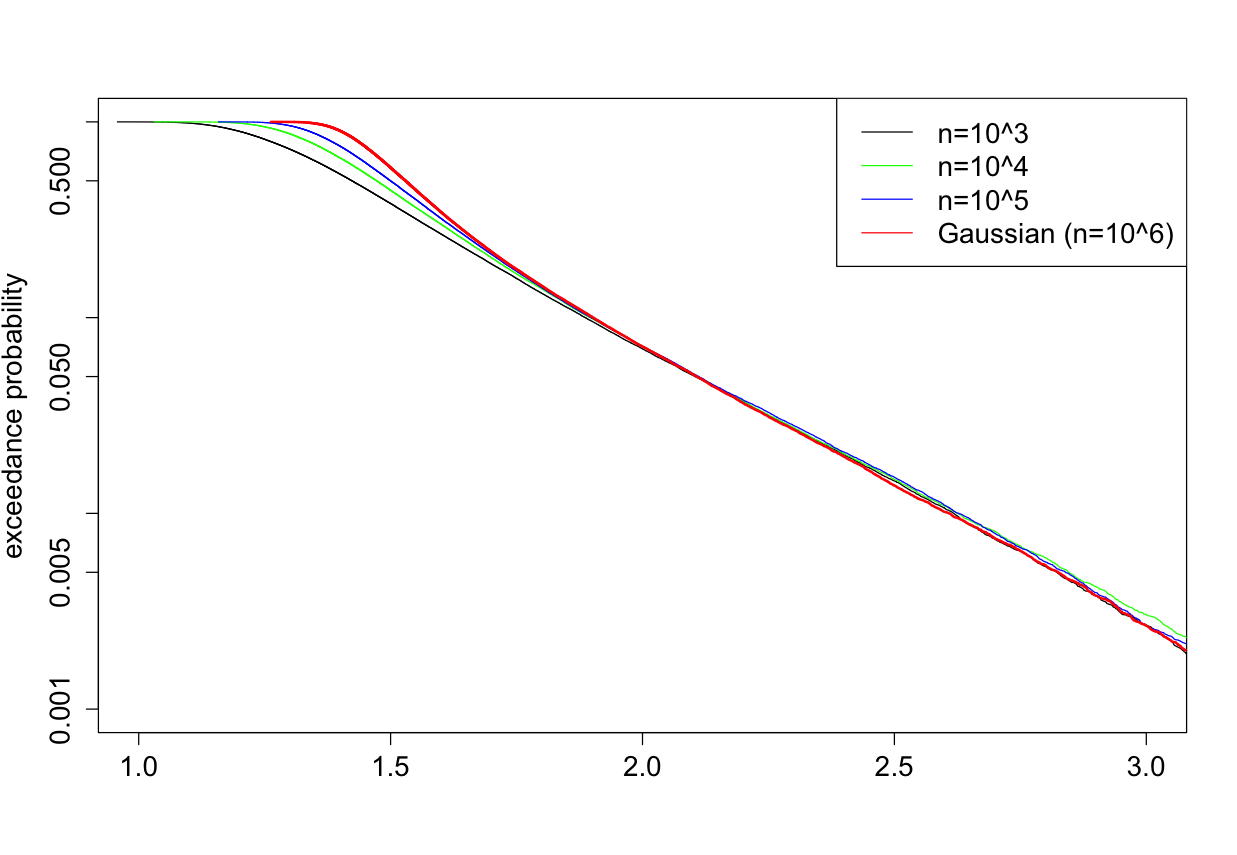}
    \includegraphics[width=0.49\textwidth]{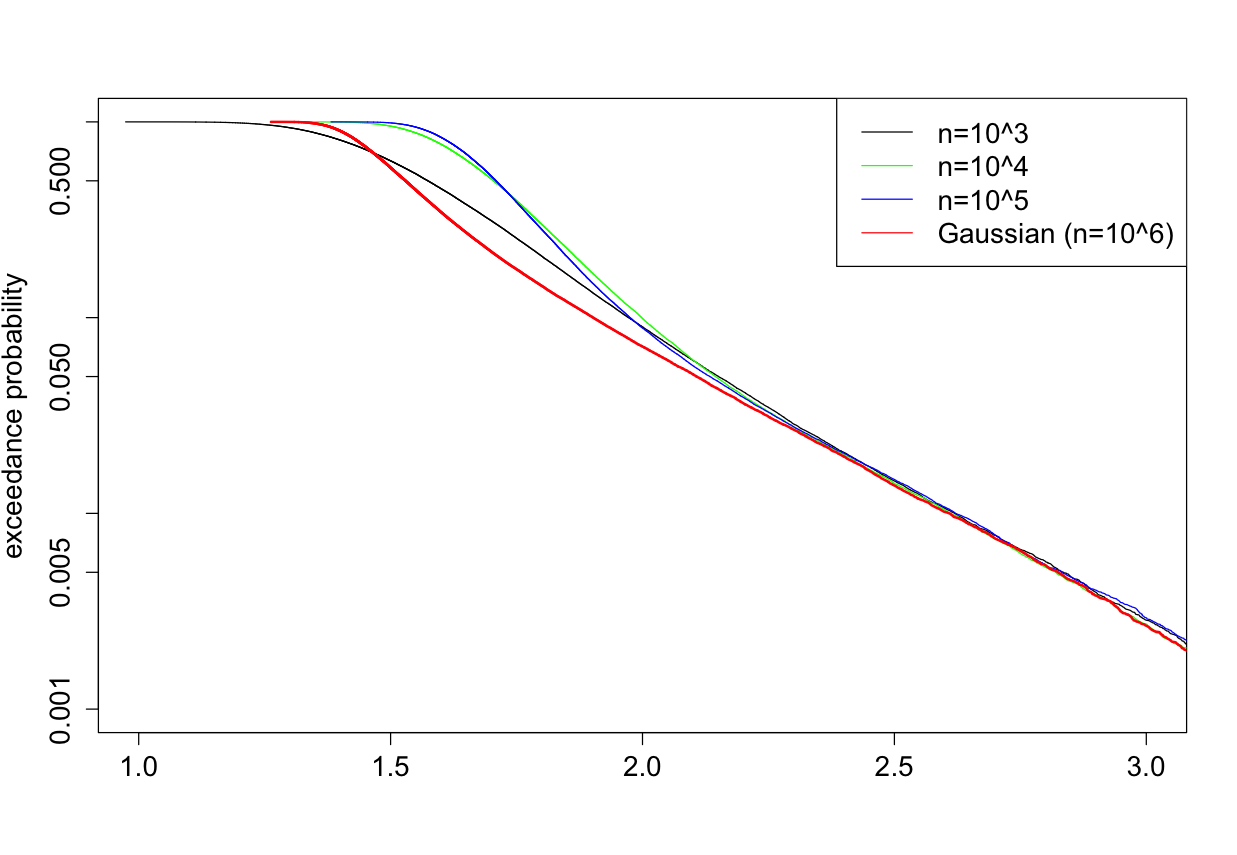}
    \caption{Survival functions of $|S_n|_{\rho_2, \mathcal{G}_{\text{dyadic}}}$ for different sample sizes with noise distribution $U(-\sqrt{3},\sqrt{3})$ (left) and $\frac{1}{2}\mathcal{N}(0,2)+\frac{1}{2} \delta_0$ (right), compared to the Gassian counterpart. The reported distributions are based on $10^5$ simulations.}
    \label{fig:ExceedancePlot}
\end{figure}

For controlling the type-I error of our procedures, the distributional approximation of $|\widetilde{S}_n|_{\rho_2}$ by $|W|_{\rho_2}$ is essential.
We study the sparse dyadic grid $\mathcal{G}_{\text{dyadic}}$ and compare the distributions for different iid unit variance innovations (a) $\eta_t^{(a)} \sim \mathcal{N}(0,1)$, (b) $\eta_t^{(b)} \sim U(-\sqrt{3},\sqrt{3})$, or (c) $\eta_t^{(c)} \sim \frac{1}{2} \mathcal{N}(0,2) + \frac{1}{2} \delta_0$ as in Proposition \ref{prop:invalid}. 
Figure \ref{fig:ExceedancePlot} depicts the survival functions $P(|\widetilde{S}_n|_{\rho_2, \mathcal{G}_{\text{dyadic}}}>x)$ for these three cases and different sample size $n$. 
We see that the distributions are almost identical in the tail, but different in the lower part of the support. 
This is in line with our theoretical results.
In the sequel, we use the quantiles of the Gaussian case (a) with sample size $n=10^6$ as substitute for the asymptotic quantiles of $|W|_{\rho_2}$.
To analyze the effect of the sample size, we present in Table \ref{tab:LS-rejection} the type I error rates for the signal discovery test, using the asymptotic critical values.
It is found that for Gaussian noise (a), the finite sample distribution and its asymptotic counterpart are very close.
For the uniform innovations (b), the approximation is accurate for significance levels below $1\%$ at small sample sizes, and below $10\%$ for samples sizes above $n=5000$.
In the case of mixture innovations (c), the approximation is accurate for small significance level $\alpha=0.1\%$, less accurate for $\alpha=1\%$, and completely inaccurate for $\alpha=10\%$, even at large sample sizes. 
This is a consequence of the heavier tails of the noise, and in line with Proposition \ref{prop:lower-tail}.

As example (d), we study an AR(1) process $\eta_t^{(d)} = 0.3\eta_{t-1}^{(d)} +\eta_t^{(b)}$. 
Due to stationarity, we may estimate the long-run-variance via $\widehat{\Sigma}_n(1)$, and choose the bandwidth $b_n=\log_{10}(n)^2$.
The results in Table \ref{tab:LS-rejection} illustrate that a larger sample size is required for accurate test sizes. 
This can be traced back to the difficulty of estimating the long-run variance.
In particular, if we plug-in the true, but statistically unfeasible, long-run-variance, the size approximation at significance levels $\alpha=1\%$ and $\alpha=0.1\%$ improves; see case (d)*.

To analyze the approximation for nonstationary and dependent error sequences, we simulate (e) the time-varying AR(1) process $\eta_{t,n}^{(e)} = a(\frac{t}{n}) \eta_{t-1,n}^{(e)} + \sigma(\frac{t}{n})\epsilon_t^{(b)}$, with $a(u) = 0.3 u$ and $\sigma(u)=1+u$.
For the estimation of the local long-run variance, we use the window size $b=\log_{10}(n)^2$ as above, and cut-off $c_n=n^{-0.33}$ for the bootstrap. It is found that for significance levels of $\alpha=1\%$ and $\alpha=0.1\%$ the approximation is getting more accurate with increasing sample size, while for $\alpha=10\%$ the approximation stays inaccurate.
This is in line with our theoretical results on thresholded weak convergence.

\begin{table}[tb]
    \centering
    \footnotesize
    \resizebox{\columnwidth}{!}{%
    \begin{tabular}{lccccc}
  \toprule
 noise model & $n=500$ & $n=1000$ & $n=5000$ & $n=10000$ & $n=50000$\\ 
  \midrule
(a) $\mathcal{N}(0,1)$ & .0846$|$.0097$|$.0010 & .0863$|$.0102$|$.0012 & .0930$|$.0099$|$.0010 & .0973$|$.0105$|$.0011 & .1000$|$.0100$|$.0010 \\ 
  (b) $U(-\sqrt{3},\sqrt{3})$ & .0680$|$.0092$|$.0010 & .0741$|$.0094$|$.0010 & .0836$|$.0098$|$.0008 & .0863$|$.0101$|$.0009 & .0918$|$.0097$|$.0008 \\ 
  (c) $\frac{1}{2}\mathcal{N}(0,2)+\frac{1}{2}\delta_0$ & .1678$|$.0187$|$.0013 & .1898$|$.0179$|$.0012 & .3964$|$.0237$|$.0011 & .4500$|$.0233$|$.0010 & .4621$|$.0184$|$.0009 \\ 
  (d) AR(1) & .1041$|$.0191$|$.0033 & .1042$|$.0172$|$.0022 & .0981$|$.0125$|$.0014 & .0978$|$.0125$|$.0013 & .0988$|$.0106$|$.0011 \\ 
    (d)* AR(1) & .0360$|$.0086$|$.0009 &.0399$|$.0093$|$.0008 & .0423$|$.0089$|$.0010 & .0452$|$.0102$|$.0010 & .0468$|$.0096$|$.0010 \\ 
    (e) tvAR(1) &.2967$|$.0260$|$.0028&.2400$|$.0204$|$.0027&.2142$|$.0140$|$.0015&.2275$|$.0142$|$.0016&.1945$|$.0115$|$.0007\\
   \bottomrule
\end{tabular}%
}
    \caption{Type-I error for the signal discovery test based on $|S_n|_{\rho_2,\mathcal{G}_{\text{dyadic}}}$ with nominal asymptotic significance levels $(10\%~|~1\%~|~0.1\%)$. 
    Probabilities are based on $10^4$ for the tvAR(1) model, and $10^5$ simulations for the other models. }
    \label{tab:LS-rejection}
\end{table}


\subsection{Detection power}

To illustrate the procedure for changepoint localization, we simulate data with tvAR(1) noise $5\cdot \eta_{t,n}^{(e)}$ as above, and mean value given by the \textsc{blocks} signal of \cite{fryzlewicz_wild_2014}. 
Sample size is $n=2048$. 
Figure \ref{fig:example-blocks} depicts the underlying signal, one realization of the noisy observations, and the corresponding intervals of significance $\NSP(\mathcal{I}_n^\tau)$.
Here, the critical values $\widehat{q}_{\alpha,n}$ are determined based on $M=10^4$ Monte Carlo simulations, with lower bound $c_n=n^{-0.45}$ and block-size $b=10$, and modulus $\rho_{2,1000}$ to put more emphasis on shorter intervals.
The same simulation is repeated $10^4$ times, and Table \ref{tab:simulation-change-tvar} reports for each changepoint how often it has been detected, i.e.\ covered by an interval of significance, and isolated, i.e.\ uniquely covered by an interval of significance.
The latter probability is always smaller, as it may occur that one significant interval contains two or more changepoints, meaning that both are detected, but none of them is isolated.
Moreover, we report the mean length of the significant interval locating the change, conditional on it being detected.  
We observe that the changes are detected with different probabilities, and that shorter or smaller changes are generally harder to detect. 
The probability of a type I error, i.e.\ that at least one significant intervals does not contain any changepoint, is $1.8\%$.
This is much lower than the nominal level $5\%$, because for some intervals where the local test statistic $T_n^*(\F_{\text{const}}^I)$ based on pure noise would falsely exceed the threshold, we actually have a changepoint. 
In other words, because there are many changepoints, it is less probable to raise a false alarm.

\begin{figure}[tb]
    \centering
    \includegraphics[width=\linewidth]{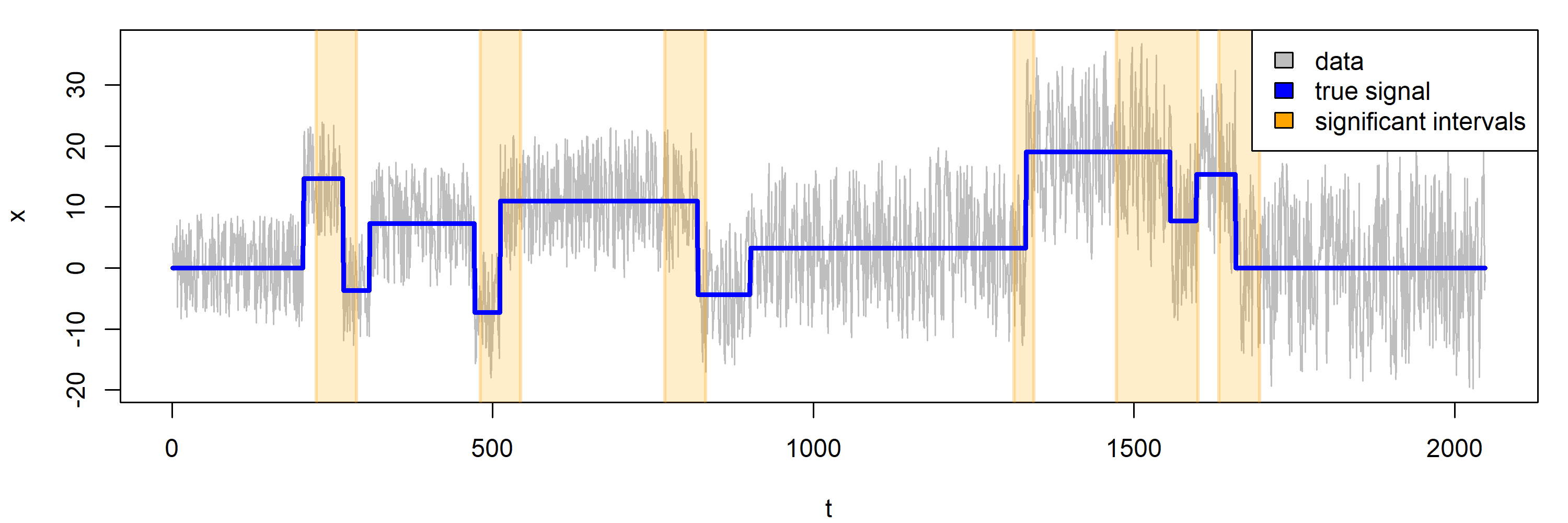}
    \caption{The \textsc{blocks} signal of \cite{fryzlewicz_wild_2014} (blue line), nonstationary noisy observations (gray), and intervals of significance for the changepoints (shaded areas).}
    \label{fig:example-blocks}
\end{figure}

\begin{table}[tb]
    \centering
    \footnotesize
    \resizebox{\columnwidth}{!}{%
    \begin{tabular}{lccccccccccc}
    \toprule
         Change location $\tau$ & 205 & 267 & 308 & 472 & 512 & 820 & 902 & 1332 & 1557 & 1598 & 1659 \\ \midrule
         Detection probability & 0.36 & 0.75 & 0.04 & 0.11 & 1.00 & 1.00 & 0.01 & 1.00 & 0.46 & 0.43 & 0.99 \\
         Isolation probability & 0.35 & 0.71 & 0.00 & 0.00 & 0.89 & 1.00 & 0.01 & 1.00 & 0.04 & 0.01 & 0.97 \\
         Mean interval length & 112.71 & 64.45 & 94.78 & 76.86 & 61.43 & 108.63 & 128.79 & 97.37 & 115.85 & 117.62 & 90.74 \\ \bottomrule
    \end{tabular}%
    }
        \caption{Detection and localization performance for the \textsc{blocks} signal with nonstationary noise sequence $5\eta_t^{(e)}$, and nominal significance level $5\%$. The actual probability of any false discovery is $1.8\%$. All reported values are based on 10000 simulations.}
    \label{tab:simulation-change-tvar}
\end{table}

The same analysis is repeated with the non-Gaussian iid error terms $5\cdot \eta_t^{(c)}$ and estimated variance. 
The results in Table \ref{tab:simulation-change-iid} show that this changepoint problem is much simpler, as evidenced by higher detection probabilities and shorter intervals of significance.
As a benchmark, we also run the original Narrowest Significance Pursuit algorithm of \cite{fryzlewicz_narrowest_2024}, with the same variance estimator. 
In this situation, based on $1000$ simulations, the type I error is found to be $100\%$.
This may be explained by the fact that \cite{fryzlewicz_narrowest_2024} assumes Gaussian errors, while our simulation has non-Gaussian errors, showing the improved robustness of the multiplicative penalty.

\begin{table}[tb]
    \centering
    \footnotesize
    \resizebox{\columnwidth}{!}{%
    \begin{tabular}{lccccccccccc}
    \toprule
         Change location $\tau$ & 205 & 267 & 308 & 472 & 512 & 820 & 902 & 1332 & 1557 & 1598 & 1659 \\ \midrule
         Detection probability & 1.00 & 1.00 & 0.79 & 0.94 & 1.00 & 1.00 & 0.92 & 1.00 & 0.99 & 0.57 & 1.00 \\
         Isolation probability & 1.00 & 1.00 & 0.79 & 0.94 & 1.00 & 1.00 & 0.92 & 1.00 & 0.54 & 0.13 & 1.00 \\
         Mean interval length & 36.20 & 24.55 & 57.02 & 35.68 & 23.04 & 34.27 & 115.85 & 32.96 & 65.03 & 77.58 & 34.48 \\ \bottomrule
    \end{tabular}%
    }
    \caption{Detection and localization performance for the \textsc{blocks} signal with noise sequence $5\eta_t^{(c)}$, and nominal significance level $5\%$. The actual probability of any false discovery is $0.8\%$. All reported values are based on 10000 simulations.}
    \label{tab:simulation-change-iid}
\end{table}

\subsection{Data example}

To demonstrate our methodology in practice, we showcase its application in monitoring power grids.
The continental European power grid operates with alternating current at a nominal frequency of 50Hz. 
Imbalances in electrical load and generation directly affect the grid frequency, as e.g.\ excessive renewable generation decreases the electrical resistance of mechanical generation plants, leading to faster rotation of the generator's rotor \citep{machowski_power_2020}.
Thus, the frequency serves as a canary for the state of the grid, and is a central input for stabilizing control mechanisms. 
Of particular interest for stability assessment of the grid is the so-called rate of change of frequency (RoCoF), i.e.\ the difference of the grid frequency form one second to the next
\citep{frigo_pmu-based_2019}. 

In Figure \ref{fig:spain-blackout}, we depict the grid frequency $f_t$ on April 28, 2025, between 10:00 and 11:00 UTC, measured by the Fraunhofer Institute for Solar Energy Systems in Freiburg, Germany, at a temporal resolution of 100ms. 
The data has been retrieved via their website \textit{energy-charts.info}.
We also depict the raw RoCoF $y_t=|f_t-f_{t-1}|$. 
At 10:33, a major power blackout on the Iberian peninsula occurred, which also affected the grid in Freiburg as can be seen from the decrease in frequency around that time. 
To determine the timing of this event statistically, we apply our changepoint localization procedure to the RoCoF time series $y_t$, with bootstrapped critical values to account for nonstationary and dependent noise. 
Based on inspection of the empirical autocorrelation function of $y_t$ (not depicted), we use the bootstrap scheme with block size $b=3\lceil \log_{10}(n)^2 \rceil =63$, and lower bound $c_n = n^{-0.33}$. 
For the multiscale statistic, we use the modulus of continuity $\rho_{2,a}$ with $a=50$.

The intervals of significance with nominal level $1\%$ are highlighted orange in Figure \ref{fig:spain-blackout}. 
A total of eight intervals are detected, as described in Table \ref{tab:spain}.
That is, with confidence $99\%$, each interval contains at least one change.
The Iberian blackout is clearly identified, and timed around 10:33:20, which is in line with preliminary investigations of the event \citep{spain2025entsoe} timing the blackout between 10:33:17 and 10:33:21. 
Moreover, we detect various further significant intervals directly preceeding the blackout, which indicate technical anomalies in the grid and might serve as early warning to prevent future blackouts.
We also highlight that no significant intervals after the blackout are detected. 
This may be explained by the fact that the Iberian grid was disconnected from the remainder of the European grid at 10:33:21. After an initial stabilization period, the frequency measurements in Freiburg are unaffected by the Iberian blackout.

\begin{table}
    \centering
    \footnotesize
    \begin{tabular}{rcl}
         \toprule  
            \multicolumn{3}{c}{Detected intervals}\\ \midrule
 10:07:30.0 & -- & 10:22:30.0 \\ 
 10:28:14.5 & -- & 10:28:28.6 \\ 
 10:30:56.2 & -- & 10:31:52.5 \\ 
 10:32:20.6 & -- & 10:32:48.7 \\ 
 10:33:19.9 & -- & 10:33:20.8 \\
 10:33:21.3 & -- & 10:33:21.7 \\ 
 10:33:22.1 & -- & 10:33:25.7 \\ 
 10:33:32.9 & -- & 10:33:33.3 \\ 
         \bottomrule
    \end{tabular}
    \caption{Intervals of significance for the changes in the RoCoF time series $y_t=|f_t-f_{t-1}|$, at  nominal significance level $1\%$.}
    \label{tab:spain}
\end{table}

\begin{figure}[tb]
    \centering
    \includegraphics[width=\linewidth]{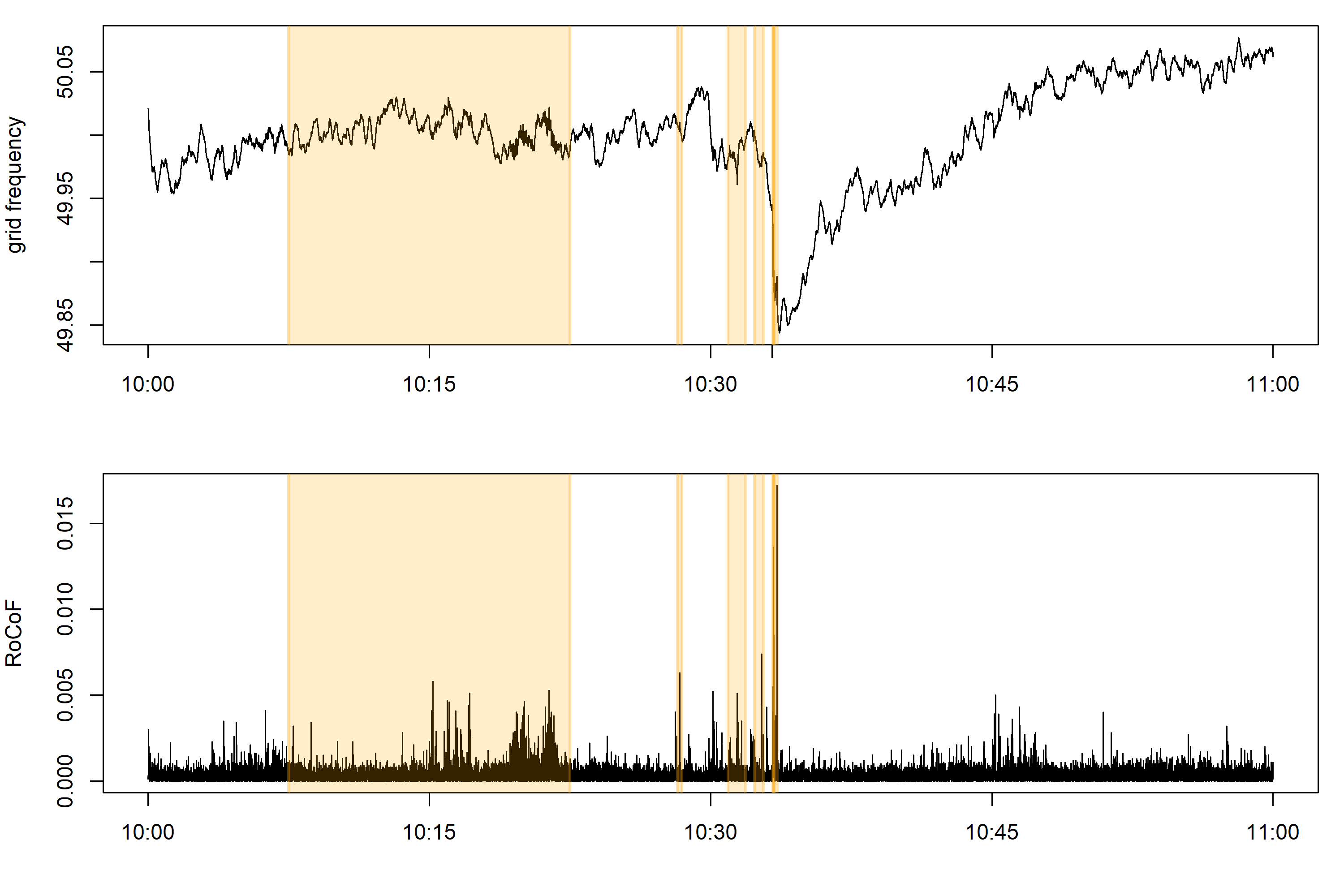}
    \caption{Top: Grid frequency $f_t$ measured in Freiburg on April 28, 2025, between 10:00 and 11:00, at 100ms resolution. Bottom: Absolute differences $y_t$ of two measurements (RoCoF). Shaded areas are significant at 95\% to contain a changepoint in RoCoF.}
    \label{fig:spain-blackout}
\end{figure}

\section*{Acknowledgements}
This work was partially supported by an IDEA league student research grant.

\appendix
\section{Proofs}

We will repeatedly make use of the following version of the Arzela-Ascoli theorem. 
For completeness, we include a proof.

\begin{theorem}[Arzela-Ascoli]
\label{Theorem:ArzelaAscoliExtension}
Let $\rho$ and $\rho'$ be moduli of continuity such that $\rho'\gg\rho$. Then the embedding $C^{\rho}\hookrightarrow C^{\rho'}_0$ is compact.
\end{theorem}
\begin{proof}[Proof of Theorem \ref{Theorem:ArzelaAscoliExtension}.]
It suffices to show that every bounded sequence in $C^{\rho}$ admits a convergent subsequence in $C^{\rho'}_0$. 
To this end, let $(x_n)_n$ be a bounded sequence in $C^{\rho}$, i.e.\ $\|x_n\|_{\rho}\leq M$ for some $M>0$. Note that, by Arzela-Ascoli, $x_n$ admits a convergent subsequence in $C[0,1]$ with respect to the usual sup-norm, since $x_n$ is uniformly continuous, as $|x_n(u)-x_n(v)|\leq M\rho(\left|u-v\right|)<\tilde\varepsilon$ for all $u,v\in[0,1]$, $\left|u-v\right|\leq\delta$, where $\delta>0$ is chosen such that $\rho(\left|u-v\right|)<\tilde\varepsilon/M$, $|u-v|\leq\delta$ for arbitrary $\tilde\varepsilon>0$. Without loss of generality, assume that $x_n\to x$ in C[0,1] for some $x\in C[0,1]$. In particular, $x_n$ is Cauchy with respect to $\left|\cdot(0)\right|$ and the sup norm in $C[0,1]$. Since $C^{\rho'}_0$ is a Banach space, it suffices to show that $x_n$ is Cauchy with respect to the seminorm $|\cdot|_{\rho'}$. 
To this end, for any $\Delta\in(0,1)$, introduce the notation
\begin{align*}
    |x|_{\rho,\geq \Delta} &= \sup_{|u-v|\geq \Delta} |x(u)-x(v)|/\rho(|u-v|), \\
    |x|_{\rho,\leq \Delta} &= \sup_{|u-v|\leq \Delta} |x(u)-x(v)|/\rho(|u-v|).
\end{align*}
Now let $\varepsilon>0$ be arbitrary.
For any $n,m\in\N$ and $\Delta>0$ such that $\rho(\left|u-v\right|)/\rho'(\left|u-v\right|)<\varepsilon/(4M)$ for all $\left|u-v\right|\leq \Delta$, we find that
\begin{align*}
    |x_n-x_m|_{\rho'}&\leq |x_n-x_m|_{\rho', \geq \Delta}+|x_n-x_m|_{\rho', \leq \Delta} \\
    &\leq \frac{\|x_n-x_m\|_{C[0,1]}}{\rho'(\Delta)}+  \left( |x_n|_{\rho, \leq \Delta}+|x_m|_{\rho, \leq \Delta} \right) \sup_{|u-v|\leq \Delta}\frac{\rho(|u-v|)}{\rho'(|u-v|)} \\
    &\leq \frac{\|x_n-x_m\|_{C[0,1]}}{\rho'(\Delta)} +\frac{\varepsilon}{4M} \left( |x_n|_{\rho} + |x_m|_{\rho}   \right) \\
    &=\frac{\|x_n-x_m\|_{C[0,1]}}{\rho'(\Delta)}+\frac{\varepsilon}{2}
\end{align*}
Thus, for $n,m\geq N$, where $N\in\N$ is such that $\|x_n-x_m\|_{C[0,1]}\leq\varepsilon\rho'(\Delta)/2$, we obtain $|x_n-x_m|_{\rho'}<\varepsilon$, i.e. $x_n$ is Cauchy with respect to $|\cdot|_{\rho'}$ and thus also with respect to $\|\cdot\|_{\rho'}$. Hence, $C^{\rho}\hookrightarrow C^{\rho'}_0$ is a compact embedding.
\end{proof}

\subsection*{Proof of Proposition \ref{prop:powerless}}
    Let $I_n \supseteq [a_n,b_n]$ such that $|I_n|\asymp h_n \geq l_n$. 
    Then $\E(T_n(I_n)) = \mu_n l_n/\sqrt{|I_n|}$ and $\Var(T_n(I_n))=\sigma^2$. 
    If $\mu_n^2 l_n \frac{l_n}{h_n} \gg \log (\frac{n}{h_n})$, then $T_n(I_n)-\sigma\sqrt{2\log ( \frac{n}{|I_n|})}\to \infty$ in probability, and thus the test is consistent.

    On the other hand, if $\mu_n^2 l_n (\frac{l_n}{h_n})= \mathcal{O}(1)$, then $\sup_{|I|\geq h_n}|\E(T_n(I))| = \mathcal{O}(1)$, and thus
    \begin{align*}
        T_{n,h_n}^{\text{DS}} \;\leq\; \sup_{|I|\geq h_n}|\E(T_n(I))| + \sup_{|I|\geq h_n} \left\{ T_n(I)-\E(T_n(I)) - \sigma\sqrt{2 \log(n/|I|)} \right\}_+ \;=\;\mathcal{O}_P(1),
    \end{align*}
    as the second term is stochastically bounded \citep{dumbgen_multiscale_2001}.

\subsection*{Proof of Theorem \ref{thm:tail-convergence}}

We make use of the following lemma, which is adapted from \cite[Lemma 10.4]{schilling_brownian_2021}.
\begin{lemma}
\label{Lemma:SchillVarianteWeibull}
Let $\rho_\alpha(h)=\sqrt{h}\log^{1/\alpha}(1/h)$. Then for all $\kappa\in\R$ and $h<1/2$
\begin{align*}
    \rho_\alpha(2^\kappa h)\leq\left[\left|\kappa\right|+1\right]^{\frac{1}{\alpha}}\sqrt{2^\kappa}\rho_\alpha(h),
\end{align*}
i.e., weights of the form $\rho_\alpha$ satisfy condition \eqref{eqn:rho-fac} in Theorem \ref{thm:tail-convergence}.
\end{lemma}
\begin{proof}[Proof of Lemma \ref{Lemma:SchillVarianteWeibull}]
        We have for $h>0$ such that $\log(2)<\left|\log(h)\right|$, i.e. $h<1/2$,
    \begin{align*}
        \frac{\rho_\alpha^\alpha(2^\kappa h)}{\rho_\alpha^\alpha(h)}=\frac{(2^\kappa)^\frac{
    \alpha}{2}\left|\log(2^\kappa h)\right|}{\left|\log(h)\right|}
    \leq\frac{(2^\kappa)^\frac{
    \alpha}{2}\left(\left|\log(2^\kappa)\right|+\left|\log(h)\right|\right)}{\left|\log(h)\right|}\leq (2^\kappa)^\frac{
    \alpha}{2} (\left|\kappa\right|+1),
    \end{align*}
    and thus $\rho_\alpha(2^\kappa h)\leq\sqrt{2^\kappa}(\left|\kappa\right|+1)^{\frac{1}{\alpha}}\rho_\alpha(h)$.
\end{proof}

\begin{theorem}\label{thm:metric}
    Let $W_n$ be a sequence of stochastic processes taking values in a metric space $(\mathcal{M},d)$, and with paths in $C[0,1]$, and $\rho_0:[0,1]\to[0,\infty)$ an increasing function with $\rho_0(0)=0$.
    Suppose there exist $C>0$ and for any $t>C$ a $\kappa(t)>1$ and $K(t)>0$, such that
    \begin{align}
        P\left(\frac{d(W_n(u),W_n(v))}{\rho_0(|u-v|)} >t\right) \leq K(t) |u-v|^{\kappa(t)},\qquad u,v\in[0,1]. \label{eqn:tail-tight-metric}\tag{T-M}
    \end{align}
    Assume further that there exist $p>0$ and $\widetilde{K}>0$ such that 
    \begin{align}
        \rho_0(zh)\leq \widetilde{K} z^p \rho_0(h), \qquad h,z\in[0,1].\label{eqn:rho-fac-2}\tag{R}
    \end{align}
    Then, for any $t>C$,
    \begin{align}
        \sup_n P\left( \sup_{u,v\in [0,1], |u-v|\leq h} \frac{d(W_n(u),W_n(v))}{\rho_0(|u-v|)} > t \right) \longrightarrow 0, \quad \text{as } h\to 0, \label{eqn:cont-modulus-metric}
    \end{align}
\end{theorem}
\begin{proof}[Proof of Theorem \ref{thm:metric}]
    We establish \eqref{eqn:cont-modulus-metric} analogously to the proof of Levy's continuity theorem \cite[Thm.~10.6]{schilling_brownian_2021}.
    Let $t>C$ be arbitrary and fix some small $\delta>0$ and define the event $A_m$ and the probability $R_m$ as 
    \begin{align}
        A_m &= \left\{ \max_{l=1,\ldots, 2^{\lfloor m\delta\rfloor}} \max_{j=0,\ldots, 2^{m}-l} \frac{d(W_n(\tfrac{l+j}{2^m}),W_n(\tfrac{j}{2^m}))}{\rho_0(\tfrac{l}{2^{m}})} \,\leq\,t \right\},\\
        R_m =P(A_m^c)&=  P\left(\max_{l=1,\ldots, 2^{\lfloor m\delta\rfloor}} \max_{j=0,\ldots, 2^{m}-l} \frac{d(W_n(\tfrac{l+j}{2^m}),W_n(\tfrac{j}{2^m}))}{\rho_0(\tfrac{l}{2^{m}})} \,>\,t\right) \\
        & \leq \sum_{l=1}^{2^{\lfloor m\delta\rfloor}} \sum_{j=0}^{2^m-l}P\left(d(W_n(\tfrac{l+j}{2^m}),W_n(\tfrac{j}{2^m}))\,>\, t\rho_0(\tfrac{l}{2^{m}}) \right) \\
        &\leq  2^m \sum_{l=1}^{2^{\lfloor m\delta\rfloor}} \left(\frac{l}{2^m}\right)^{\kappa(t)} \\
        &\leq 2^{m}2^{\lfloor m\delta\rfloor} \left(\frac{2^{\lfloor m\delta\rfloor}}{2^m}\right)^{\kappa(t)} \approx 2^{m(1+\delta)-m(1-\delta)\kappa(t)},
    \end{align}
    which is summable for any $\kappa(t)>1$, choosing $0<\delta<[\kappa(t)-1]\cdot[1+\kappa(t)]^{-1}$.\\
    Now we apply a chaining argument as follows: Denote $T_k=\{ j 2^{-k}\,|\, j=0,\ldots, 2^k \}$. 
    For any $0\leq u \leq v\leq 1$ with $|u-v|\leq h$, choose $m\in\N$ such that $2^{- (m+1)(1-\delta)}<h \leq 2^{-m(1-\delta)}$. 
    We can then find $u_m, v_m \in T_m$ such that $u\leq u_m \leq v_m \leq v$, and sequences $u_k,v_k\in T_k$, $k\geq m$, such that 
    \begin{align*}
        u_k\downarrow u, \qquad |u_k-u_{k+1}| \leq 2^{-k-1}, \\
        v_k\uparrow v,\qquad |v_k-v_{k+1}| \leq 2^{-k-1}.
    \end{align*}
    This construction yields, on the event $\underline{A}_m = \bigcap_{k=m}^\infty A_k$ with probability $P(\underline{A}_m)\to 1$,
    \begin{align*}
        &\quad \sup_{|u-v|\leq h} d(W_n(u),W_n(v)) \\
        &\leq \sup_{|u-v|\leq h} d(W_n(u_m),W_n(v_m)) + \sum_{k=m}^\infty \left[d(W_n(u_k),W_n(u_{k+1})) + d(W_n(v_k),W_n(v_{k+1}))\right] \\
        &\leq t\rho_0(u_m-v_m) + 2t \sum_{k=m+1}^\infty \rho_0(2^{-k}) \\
        &= t\rho_0(u_m-v_m) + 2t \sum_{k=0}^\infty \rho_0(2^{-k}2^{-(m+1)(1-\delta)} 2^{(m+1)\delta}) \\
        &\leq t\rho_0(u_m-v_m) + 2t \sum_{k=0}^\infty \rho_0(2^{-k} h 2^{(m+1)\delta}) \\
        & \leq t\rho_0(h) + 2t \sum_{k=m+1}^\infty \rho_0(2^{-k(1-\delta)} h ) \\
        & \overset{*}{\leq} t\rho_0(h) + 2t  \sum_{k=m+1}^\infty  \tilde{K}2^{-k(1-\delta)p}\rho_0(h) \\
        & \leq t \rho_0(h) \left[ 1+2\epsilon_m \right],
    \end{align*}
    for some sequence $\epsilon_m\to 0$.
    At the step $(*)$, we use \eqref{eqn:rho-fac-2}. Now let $M\in \N$ such that $\Delta \leq 2^{- M(1-\delta)}$.
    Since $\underline{A}_M\subset \underline{A}_m$ for all $m\geq M$, we may conclude that in the event $\underline{A}_M$,
    \begin{align*}
        \sup_{|u-v|\leq h} d(W_n(u),W_n(v)) \leq t \rho_0(h)[1+\epsilon_M], \qquad \forall h\leq \Delta.
    \end{align*}
    Hence, for any $\epsilon>0$ and $\Delta$ small enough such that $2\epsilon_{M}=2\epsilon_{M(\Delta)}<\epsilon$, 
    \begin{align*}
        P\left(\sup_{|u-v|\leq \Delta} \frac{d(W_n(u),W_n(v))}{\rho_0(|u-v|)} \,>\,t(1+\epsilon)\right) 
        &\leq P(\underline{A}_M^c),
    \end{align*}
    which tends to zero as $M\to \infty$, i.e.\ as $\Delta\to 0$.
    As $\epsilon>0$ and $t>C$ are arbitrary, we actually obtain for any $t>C$ that
    \begin{align*}
        Q(t,\Delta) := \sup_n P\left(\sup_{|u-v|\leq \Delta} \frac{d(W_n(u),W_n(v))}{\rho_0(|u-v|)} \,>\,t\right) \to 0, \qquad \text{as}\quad \Delta\to 0.
    \end{align*}
    This establishes \eqref{eqn:cont-modulus}.
\end{proof}

Now proceed to the proof of Theorem \ref{thm:tail-convergence}.
First, observe that \eqref{eqn:cont-modulus} is a consequence of \eqref{eqn:cont-modulus-metric} for the metric space $(\R, |\cdot|)$.

    From \eqref{eqn:cont-modulus}, we may conclude that $\|W_n\|_{\rho_0}$ is stochastically bounded as follows. 
    For any $N\in \N$ and any $t>6$,
    \begin{align*}
        P\left( \|W_n\|_{\rho_0}>t \right) 
        &\leq Q(\tfrac{t}{3}, \tfrac{1}{N}) + \sum_{i,j=1}^N P\left( |W_n(\tfrac{i}{N})-W_n(\tfrac{j}{N})| > \tfrac{t}{3} \rho_0(\tfrac{1}{N}) \right) \\
        &\leq Q(2, \tfrac{1}{N}) + N^2 \max_{i}P\left( |W_n(\tfrac{i}{N})| > \tfrac{t}{6} \rho_0(\tfrac{1}{N}) \right).
    \end{align*}
    The second term tends to zero as $t\to\infty$, uniformly in $n$, because the finite dimensional distributions of $W_n$ converge.
    For any $\epsilon>0$, we may choose $N=N(\epsilon)$ such that the first term is smaller than $\frac{\epsilon}{2}$, and then $t=t(N(\epsilon),\epsilon)=t(\epsilon)$ big enough such that the second term is smaller than $\frac{\epsilon}{2}$. 
    Thus, $\|W_n\|_{\rho_0}$ is stochastically bounded.
    By virtue of Theorem \ref{Theorem:ArzelaAscoliExtension}, this boundedness also implies tightness in $C^\rho_0$ for any $\rho\ll \rho_0$, and thus establishes the first claim of \eqref{eqn:wconv}.

    Regarding the second claim of \eqref{eqn:wconv}, i.e.\ the thresholed weak convergence, observe that for any $\Delta>0$ and $\rho$ such that $\rho(h)/\rho_0(h)\to 1$ as $h\to 0$,
    \begin{align*}
        P(\Delta,W_n) &\leq P\left(\|W_n\|_{\rho}>t\right) \leq P(\Delta, W_n) + R(\Delta,W_n), \\
        \text{where}\quad 
        P(\Delta,W_n) &= P\left(\sup_{|u-v|\geq \Delta} \frac{|W_n(u)-W_n(v)|}{\rho(|u-v|)} \,>\,t\right), \\
        R(\Delta,W_n) &= P\left(\sup_{|u-v|\leq \Delta} \frac{|W_n(u)-W_n(v)|}{\rho(|u-v|)} \,>\,t\right)
    \end{align*}
    By virtue of the weak convergence $W_n\wconv W$ in $C[0,1]$ established above, we conclude that $P(\Delta, W_n)\to P(\Delta, W)$ for any $\Delta>0$ as $n\to \infty$. 
    Moreover, \eqref{eqn:cont-modulus} yields $\lim_{\Delta\to 0} \sup_n R(\Delta, S_n) = 0$, also for $\rho$, since for arbitrary $\tilde\varepsilon>0$ and small enough $\tilde\Delta>0$ it holds that
    \begin{equation*}
        R(\tilde\Delta,W_n)\leq P\left(\sup_{|u-v|\leq\tilde\Delta}\frac{|W_n(u)-W_n(v)|}{\rho_0(|u-v|)}>t(1-\tilde\varepsilon)\right).
    \end{equation*}
    Thus the claim holds for all $t>C/(1-\tilde\varepsilon)$, and since $\tilde\varepsilon>0$ can be chosen arbritrary, the second claim of \eqref{eqn:wconv} follows for all $t>C$ as desired. This completes the proof of Theorem \ref{thm:tail-convergence}.

\subsection*{Proof of Theorem \ref{thm:Skorokhod-tail}}
    The boundedness in $C^{\rho^*}$ implies tightness in $C[0,1]$ (Theorem \ref{Theorem:ArzelaAscoliExtension}), and thus $W_n\wconv W$ in $C[0,1]$. 
    By Skorokhod's representation theorem, upon potentially changing the probability space, $\| W_n - W\|_{C[0,1]}\pconv 0$.
    Next, for any $\Delta>0$ and any stochastic process $X$, denote the quantities
    \begin{align*}
        \|X\|_{\rho^*, \geq \Delta} &= \sup_{|u-v|\geq \Delta} \frac{|X(u)-X(v)|}{\rho^*(|u-v|)}, \\
        \|X\|_{\rho^*, \leq \Delta} &= \sup_{|u-v|\leq \Delta} \frac{|X(u)-X(v)|}{\rho^*(|u-v|)},
    \end{align*}
   and  $\|X\|_{\rho, \geq \Delta}$ and  $\|X\|_{\rho, \leq \Delta}$ are to be understood in the same way. 
   Then
    \begin{align*}
        \|W_n - W\|_{\rho} &\leq \|W_n-W\|_{\rho, \geq \Delta} + \|W_n-W\|_{\rho, \leq \Delta}.
    \end{align*}
    For any fixed $\Delta>0$, the first term tends to zero in probability as $n\to \infty$ because $\|W_n-W\|_{C[0,1]}\pconv 0$. For the right hand term, note that
    \begin{align*}
        \|W_n-W\|_{\rho,\leq \Delta} &\leq \|W_n\|_{\rho,\leq\Delta}+ \|W\|_{\rho,\leq\Delta} \\
        &= \left(\sup_{h\leq \Delta} \frac{\rho^*(h)}{\rho(h)}\right)\left( \|W_n\|_{\rho^*, \leq \Delta} + \|W\|_{\rho^*,\Delta}\right)
        &\leq C(\Delta) \left( \|W_n\|_{\rho^*} + \|W\|_{\rho^*} \right),
    \end{align*}
    where $C(\Delta)\to 0$, as $\Delta\to 0$, since $\rho\gg\rho^*$. 
    Since $\|W_n\|_{\rho^*}+\|W\|_{\rho^*}$ is stochastically bounded, for any $\varepsilon>0$ and $\delta>0$, there are $N\in\N$ and $\tilde\Delta>0$ such that
    \begin{equation*}
        P(\|W_n-W\|_{\rho,\leq \Delta}>\varepsilon)<\delta,\,n\geq N,\,0<\Delta\leq \tilde\Delta.
    \end{equation*}
    Thus, $\|W_n - W\|_{\rho}\pconv 0$ follows for every $\rho\gg\rho^*$.

\subsection*{Proof of Proposition \ref{prop:lower-tail}}
    For arbritrary $T>0$, let $\sigma>0$ and $p\in (0,1)$ be such that $1/(\sigma\sqrt{p})>T$. Let $\epsilon_t$ be iid symmetric random variables with $P(|\epsilon_t|>r) = \exp(-r^2)$, and $\sigma^2=\Var(\epsilon_t)$. 
    For any $p\in(0,1)$, let $\xi_t\iid \text{bin}(1,p)$ and define $\eta_t = \epsilon_t \xi_t/\sqrt{\sigma^2 p}$, such that $\Var(\eta_t)=1$.
    Define the interpolated partial sum process 
    \begin{align*}
        S_n(u) = \frac{1}{\sqrt{n}} \sum_{t=1}^{\lfloor nu\rfloor}\eta_t + \sqrt{n}\left[u-\tfrac{\lfloor un\rfloor}{n}\right]\eta_{\lceil un\rceil}.
    \end{align*}
    Then $\|S_n\|_{\rho_2} \geq \max_{t=1,\ldots, n} |\eta_t|/\sqrt{1+\log n}$.
    Hence,
    \begin{align*}
        P\left( \|S_n\|_{\rho_\alpha} \leq c \right) 
        &\leq P\left(|\eta_1|  \leq c (1+\log n)^{\frac{1}{2}} \right)^n \\
        &=\left(1-p + p \exp\left(-\left( c \sigma \sqrt{p} \right)^2 (1+\log n)\right) \right)^n \\
        &\leq\left(1 - p n^{-\left( \frac{c}{\sigma \sqrt{p}} \right)^2} \right)^n.
    \end{align*}
    For $c<1/\sqrt{\sigma^2 p}$, the latter term tends to zero as $n\to \infty$. 
    This shows that $P(\|S_n\|_{\rho_2}\geq \frac{1}{\sigma \sqrt{p}}) \to 1$, proving Proposition \ref{prop:lower-tail} by the specific choice of $\sigma$ and $p$.
.

\subsection*{Proof of Theorem \ref{thm:concentration-dependent}}
    This proof is inspired by \cite[Thm.~1]{Liu2013}.
    Let $S_n = \sum_{t=1}^n w_t\eta_t$, and 
    \begin{align*}
        \eta_{t,j} &= \E(\eta_t|\epsilon_t, \epsilon_{t-1},\ldots, \epsilon_{t-j}), \\
        S_{n,j} &= \sum_{t=1}^n w_t \eta_{t,j},\\
        Y_{i,j} &= \sum_{t=(i-1)j+1}^{(ij)\wedge n} w_t( \eta_{t,j}-\eta_{t,j-1}),\qquad i=1,\ldots, \lfloor n/j\rfloor +1.
    \end{align*}
    By telescoping, we find that 
    \begin{align}
        \|S_{n}\|_{\psi_2} &\leq \sum_{j=1}^\infty \| S_{n,j}-S_{n,j-1}\|_{\psi_2} \\
        &\leq \sum_{j=1}^\infty \left\| \sum_{i~\text{is odd}} Y_{i,j} \right\|_{\psi_2} + \sum_{j=1}^\infty\left\| \sum_{i~\text{is even}} Y_{i,j} \right\|_{\psi_2}. \label{eqn:concentration1}
    \end{align}
    Observe that the $Y_{1,j}, Y_{3,j},\ldots$ are independent by construction, and the same holds for the even indices. 
    Hence, using that $\|\eta_{t,j}-\eta_{t,j-1}\|_{\psi_2} \leq \delta(j)$,
    \begin{align}
        \left\| \sum_{i~\text{is odd}} Y_{i,j} \right\|_{\psi_2} 
        &\leq \tilde K \sqrt{\sum_i \|Y_{i,j}\|_{\psi_2}^2} \\
        &\leq \tilde K \sqrt{\sum_i \left(\sum_{t=(i-1)j+1}^{(ij)\wedge n} |w_t| \delta_{\psi_2}(j)\right)^2 } \quad \\
        &\leq \tilde K \sqrt{\sum_i \left(\sum_{t=(i-1)j+1}^{(ij)\wedge n} |w_t|^2 \right) (j\delta_{\psi_2}(j)^2) } \\
        &= \tilde K \sqrt{j} \delta_{\psi_2}(j) \sqrt{\sum_{t=1}^n |w_t|^2}.\label{eqn:concentration2}
    \end{align}    
    for some universal constant $\tilde K>0$, stemming from the sub-Gaussian concentration bound. The same bound holds for the even indices $i$. 
    Combining \eqref{eqn:concentration1} and \eqref{eqn:concentration2} yields the result for $K=2\tilde K$.

\subsection*{Proof of Corollaries \ref{cor:Donsker} and \ref{cor:timeseries}}

\begin{proof}[Proof of Corollary \ref{cor:Donsker}]
    We just need to verify the tightness condition \eqref{eqn:tail-tight}.
    For arbitrary $0\leq u<v\leq 1$, we may write $S_n(u)-S_n(v) =\frac{1}{\sqrt{n}} \sum_{t=1}^{n} \eta_t w_n(t,u,v)$ for weights $w_n(t,u,v)\in[0,1]$ such that $\sum_{t=1}^n w_n(t,u,v)^2\leq\sum_t w_n(t,u,v) = |v-u|$. 
    Then \eqref{eqn:tail-tight} is a consequence of Hoeffding's inequality.
\end{proof}

\begin{proof}[Proof of Corollary \ref{cor:timeseries}]
    As in Corollary \ref{cor:Donsker}, we obtain \eqref{eqn:tail-tight} as a consequence of Theorem \ref{thm:concentration-dependent}. The convergence of the finite dimensional marginals may be obtained, for example, via Theorem 5 of \cite{mies_strong_2024}.
\end{proof}

\subsection*{Proof of Theorem \ref{thm:signal}}
    Under the null, $\widehat{\sigma}_n^2$ is consistent as an average of a $1$-dependent sequence. Corollary \ref{cor:Donsker} yields $(T_n^*\vee \tau) \wconv (\sigma |B|_{\rho_2}\vee \tau)$ for any $\tau>\sigma C_\eta$. 
    By Slutsky's Lemma and the consistency of $\widehat{\sigma}_n^2$, we find $(\frac{T_n^*}{\widehat{\sigma}_n} \vee \tau) \wconv (|B|_{\rho_2}\vee \tau)$ for any $\tau>C_\eta$, which yields the first claim.

    Under the alternative, $\widehat{\sigma}_n^2$ is still consistent because
    \begin{align*}
        \widehat{\sigma}_n^2 = \frac{1}{2(n-1)}\sum_{t=2}^n (\eta_t-\eta_{t-1}+f(t)-f(t-1))^2 = \frac{1}{2n}\sum_{t=2}^n (\eta_t-\eta_{t-1})^2 + \mathcal{O}_P(\mu_n^2/n).
    \end{align*}
    Moreover, $T_n^* \geq T_n(I_n)/\sqrt{\log \frac{en}{|I_n|}}\to \infty$ in probability. 
    The latter convergence holds because $\Var(\sum_{t\in I_n} Y_t / \sqrt{|I_n|}) =\sigma$ whereas $\E(\sum_{t\in I_n} Y_t / \sqrt{|I_n|}) = \sqrt{|I_n|} \mu_n \to \infty$, which implies $T_n(I_n) \asymp \sqrt{|I_n|}\mu_n \gg \sqrt{\log \frac{en}{|I_n|}}$ by assumption.

\subsection*{Proof of Proposition \ref{prop:DS-A}}
    For $A> C_\eta$, Corollary \ref{cor:Donsker} and \eqref{eqn:cont-modulus} yield for any $t>0$
    \begin{align*}
        \lim_{h\downarrow 0} \sup_{n} P\left( \sup_{|I|\leq h} \{ T_n(I)-A\sqrt{\log \tfrac{en}{|I|}} \}_+ > t \right) = 0.
    \end{align*}
    Moreover, 
    \begin{align*}
        \sup_{|I|>h} \left\{T_n(I) - A \sqrt{\log \tfrac{en}{|I|}}\right\}_+ &\;\underset{n\to\infty}{\wconv}\; \sup_{\substack{u\leq v \\ |u-v|>h}}\left\{ \frac{|B(v)-B(u)|}{\sqrt{v-u}} - A \sqrt{\log \tfrac{e}{|u-v|}} \right\}_+ \\
        &\;\underset{h\to 0}{\wconv}\; \sup_{u\leq v }\left\{ \frac{|B(v)-B(u)|}{\sqrt{v-u}} - A \sqrt{\log \tfrac{e}{|u-v|}} \right\}_+ .
    \end{align*}
    Standard arguments yield the weak convergence of the full statistic \citep[Thm.~3.2]{Billingsley1999}.

    The second claim is a direct consequence of Proposition \ref{prop:lower-tail}.

\subsection*{Proof of Theorem \ref{thm:gof}}
    Since $T_n^*(\Fhyp_0)\leq T_n^*(f_n)$ under $\mathbb{H}_0$, the proof is identical to Theorem \ref{thm:signal} if we can show that $\widehat{\sigma}^2_n \to \sigma^2$ in probability.
    To this end, we observe that
    \begin{align*}
        &\left| \widehat{\sigma}_n^2 - \frac{1}{2(n-1)}\sum_{t=2}^n (\eta_t-\eta_{t-1})^2 \right| \\
        &\leq \frac{1}{2(n-1)} \sum_{t=2}^n 2|\eta_t-\eta_{t-1}|\cdot |f_n(t)-f_n(t-1)| + \frac{1}{2(n-1)} \sum_{t=2}^n |f_n(t)-f_n(t-1)|^2 \\
        &\leq \sqrt{\frac{1}{(n-1)} \sum_{t=2}^n |\eta_t-\eta_{t-1}|^2} \sqrt{\frac{1}{(n-1)} \sum_{t=2}^n |f_n(t)-f_n(t-1)|^2} + o(1)\\
        &= \mathcal{O}_P(1) \cdot o(1) + o(1),
    \end{align*}
    where the first term is bounded by the law of large numbers. 
    This establishes the consistency of $\widehat{\sigma}_n^2$.

\subsection*{Proof of Proposition \ref{prop:localized}}
    We observe that
    \begin{align*}
        P\left( \exists I\in \mathcal{I}_n\text{ such that } f_n\in \Fhyp_0^I \right)
        &\leq P\left( \bigcup_{I:f_n\in \Fhyp_0^I} \{ T_n^*(\Fhyp_0^I) > \widehat{\sigma}_n \cdot q_\alpha \} \right) \\
        &\leq P\left( \sup_{I:f_n\in \Fhyp_0^I} T_n^*(\Fhyp_0^I) > \widehat{\sigma}_n \cdot q_\alpha \right) \\
        &\leq P\left(  T_n^*(\{f_n\}) > \widehat{\sigma}_n \cdot q_\alpha \right),
    \end{align*}
    which is asymptotically less than $\alpha$.

\subsection*{Proof of Theorem \ref{thm:changepoint}}
    The first claim is a consequence of Proposition \ref{prop:localized}.
    For the power statement, let $k$ such that $\tau_k\in \mathcal{D}(z)$ and set $I_n(\tau_k) = [\tau_k-\frac{w_k}{2}, \tau_k+\frac{w_k}{2}]$ for $w_k=w(\tau_k,z)$ such that $I_n(\tau_k) \cap \{\tau_1,\ldots, \tau_\kappa\} = \{\tau_k\}$.
    Then
    \begin{align*}
        T_n^*\left(\Fhyp_{\text{const}}^{I_n(\tau_k)}\right)
        &= \inf_{a\in\R} \sup_{[un,vn]\subset I_n(\tau_k)} \frac{\left|S_n(v)-S_n(u) - [v-u]a\right|}{\rho_2(v-u)}.
    \end{align*}
    Denote $\mu_k = \sum_{j=0}^{k-1} \delta_j = f_n(\tau_k-1)$, and assume without loss of generality that $\delta_k>0$.
    For any $a\in\R$, we have $a/\sqrt{n}\leq \mu_k + \frac{\delta_k}{2}$ or $a/\sqrt{n}\geq \mu_k + \frac{\delta_k}{2}$. 
    Assume for now the latter case. 
    Then 
    \begin{align*}
        &\sup_{[un,vn]\subset I_n(\tau_k)} \frac{\left|S_n(v)-S_n(u) - [v-u]a\right|}{\rho_2(v-u)} \\
        &\geq \frac{\left|S_n((\tau_k+\frac{w_k}{2})/n)-S_n(\tau_k/n) - \frac{w_k a}{2n}\right|}{\rho_2(\frac{w_k}{2n})} \\
        &= \frac{\left| \frac{w_k}{2\sqrt{n}} (\mu_k+\delta_k)+\tilde{S}_n((\tau_k+\frac{w_k}{2})/n)-\widetilde{S}_n(\tau_k/n) - \frac{w_k a}{2n}\right|}{\rho_2(\frac{w_k}{2n})}\\
        &\geq \frac{\frac{w_k \delta_k}{4\sqrt{n}}}{\rho_2(\frac{w_k}{2n})}  - |\widetilde{S}_n|_{\rho_2}\\
        &\geq c\frac{\sqrt{w_k} \delta_k}{\sqrt{\log \frac{en}{2w_k} }} - |\widetilde{S}_n|_{\rho_2}
        \quad \geq \quad c\sqrt{z} - |\widetilde{S}_n|_{\rho_2},
    \end{align*}
    for some small $c>0$.
    If instead $a/\sqrt{n} \leq \mu_k + \frac{\delta_k}{2}$, we observe that
    \begin{align*}
        &\sup_{[un,vn]\subset I_n(\tau_k)} \frac{\left|S_n(v)-S_n(u) - [v-u]a\right|}{\rho_2(v-u)} \\
        &\geq \frac{\left|S_n(\tau_k/n)- S_n((\tau_k-\frac{w_k}{2})/n) - \frac{w_k a}{2n}\right|}{\rho_2(\frac{w_k}{2n})} \\
        &\geq \frac{\left|\widetilde{S}_n(\tau_k/n)- \widetilde{S}_n((\tau_k-\frac{w_k}{2})/n) + \frac{w_k \mu_k}{2\sqrt{n}} - \frac{w_k a}{2n}\right|}{\rho_2(\frac{w_k}{2n})} \\
        &\geq \frac{\frac{w_k \delta_k}{4\sqrt{n}}}{\rho_2(\frac{w_k}{2n})}  - |\widetilde{S}_n|_{\rho_2}
        \quad \geq \quad c\sqrt{z} - |\widetilde{S}_n|_{\rho_2}.
    \end{align*}
    By virtue of Theorem \ref{thm:tail-convergence} and Corollary \ref{cor:timeseries}, the random variable $|\widetilde{S}_n|$ is stochastically bounded, and thus $T^*(\Fhyp_{\text{const}}^{I_n(\tau_k)})\to \infty$ as $z\to\infty$, showing that $I_n(\tau_k)\in \mathcal{I}_n$ eventually. 
    Note furthermore that $I_n(\tau_k)\cap \{ \tau_1,\ldots, \tau_\kappa\} = \{\tau_k\}$.
    Thus,
    \begin{align*}
        &\quad P\left( \forall \tilde{\tau}\in \mathcal{D}(z)\; \exists I\in \mathcal{I}_n^\tau\text{ such that } I \cap \{\tau_1,\ldots, \tau_\kappa\} = \{\tilde{\tau}\} \right) \\
        &\geq P\left( \forall \tau_k\in D(z)\,\,:\, T^*(\Fhyp_{\text{const}}^{I_n(\tau_k)}) \geq \widehat{\sigma}_n\cdot q_\alpha \right)\\
        &\geq P\left( c \cdot \sqrt{z} \;>\; \widehat{\sigma}_n\cdot q_\alpha + |\widetilde{S}_n|_{\rho_2} \right),
    \end{align*}
    which tends to one as $n\to \infty$ and $z\to \infty$.

\subsection*{Proof of Proposition \ref{prop:NSP}}
    Because any $I\in\mathcal{I}$ contains at least one point $\tau\in\widetilde{\mathcal{D}}$, we may decompose 
    \begin{align*}
        \bigcup_{\tau\in \widetilde{\mathcal{D}}} \widetilde{\mathcal{I}}(\tau) \; \subseteq\, \mathcal{I} \, \subseteq\; \bigcup_{\tau\in\widetilde{\mathcal{D}}} \mathcal{I}(\tau),
    \end{align*}
    for
    \begin{align*}
        \mathcal{I}(\tau) = \{ I\in \mathcal{I}\,|\, \tau \in I\}, \qquad \widetilde{\mathcal{I}}(\tau) = \left\{ I\in \mathcal{I}(\tau)\text{ such that } |I|\leq r(\tau)\right\} \neq \emptyset.
    \end{align*}
    Choosing one interval $\widehat{I}_1$ of shortest length, we find that $\widehat{I}_1 \in \widetilde{\mathcal{I}}(\tau_{(1)})$ for some $\tau_{(1)}\in \widetilde{\mathcal{D}}$. 
    Since $\tau_{(1)} \in \widehat{I}_1$, the class $\mathcal{I}(\tau_{(1)})$ is removed from the candidate set, while all other $\widetilde{\mathcal{I}}(\tau_{(k)})$ are still relevant as they do not contain $\tau_{(1)}$.
    Thus, at the next step, we have \[\bigcup_{\tau\in \widetilde{\mathcal{D}}\setminus\{\tau_{(1)}\}} \widetilde{\mathcal{I}}(\tau)\;\subseteq\,\mathcal{I}^{1} \,\subseteq\;  \bigcup_{\tau\in\widetilde{\mathcal{D}}\setminus\{\tau_{(1)}\}} \mathcal{I}(\tau) .\]
    Proceeding inductively, we find that $\widehat{I}_k\in \widetilde{\mathcal{I}}(\tau_{(k)})$ for some $\tau_{(k)}\in \widetilde{\mathcal{D}}$.
    This implies that the iteration stops after exactly $|\widetilde{\mathcal{I}}|$ steps, with the claimed guarantees.

\subsection*{Proof of Theorem \ref{thm:changepoint-LS}}

\begin{lemma}\label{lem:Sigmahat-LS}
    Under the conditions of Theorem \ref{thm:changepoint-LS}, 
    \begin{align*}
        \sup_{u\in[0,1]} \left| \widehat{\Sigma}_n(u) - \Sigma(u) \right| \pconv 0,
    \end{align*}
    and for some constant $K$, and all $c\in(0,1)$,
    \begin{align*}
        \sup_{|u-v|>c} \left| \widehat{\Sigma}_n(u) - \widehat{\Sigma}_n(v) \right|  
        \leq Kc + \mathcal{O}_P\left( \sqrt{\frac{v_n b_n}{n}} + b_n^{-1} \right).
    \end{align*}
\end{lemma}
\begin{proof}[Proof of Lemma \ref{lem:Sigmahat-LS}]
    It is sufficient to consider $u=\frac{i}{n}$ and $v=\frac{j}{n}$, as both claims of the Lemma readily extend via interpolation.
    
    Define the terms
    \begin{align*}
        \widetilde{\Sigma}_n(u) &= \frac{1}{2n}\sum_{t=b_n}^{\lfloor un\rfloor -b_n}  \left( \frac{1}{\sqrt{b_n}}\sum_{i=0}^{b_n-1}  (\eta_{t-i} - \eta_{t+1+i}) \right)^2, \\
        \Delta_n &= \frac{1}{2n}\sum_{t=b_n}^{n -b_n} \left( \frac{1}{\sqrt{b_n}}\sum_{i=0}^{b_n-1}  (f_n(t-i) - f_n(t+1+i)) \right)^2,
    \end{align*}
    for $u=\frac{i}{n}$, and interpolated in between.
    By expanding the square and applying the Cauchy-Schwarz inequality, we find that
    \begin{align}
        \left|\widehat{\Sigma}_n(u) - \widetilde{\Sigma}_n(u)\right|
        &\leq \Delta_n + \sqrt{\widetilde{\Sigma}_n(1) \Delta_n}.\label{eqn:intvar-err}
    \end{align}
    We will show that $\Delta_n\to 0$ and $\widetilde{\Sigma}_n(u)\to \Sigma(u)$ uniformly, which implies that \eqref{eqn:intvar-err} tends to zero.
    The term $\Delta_n$ may be bounded as
    \begin{align*}
        \Delta_n 
        &\leq \frac{1}{2n}\sum_{t=b_n}^{\lfloor un\rfloor -b_n} \left( \frac{1}{\sqrt{b_n}}\sum_{j=-b_n}^{b_n-1}  |f_n(t+j+1) - f_n(t+j)| \right)^2 \\
        &= \frac{1}{2n}\sum_{t=b_n}^{\lfloor un\rfloor -b_n} 4 b_n \left( \frac{1}{2b_n}\sum_{j=-b_n}^{b_n-1}  |f_n(t+j+1) - f_n(t+j)| \right)^2 \\
        &\leq \frac{1}{2n}\sum_{t=b_n}^{\lfloor un\rfloor -b_n} 2 \sum_{j=-b_n}^{b_n-1}  |f_n(t+j+1) - f_n(t+j)|^2\\
        &\leq \frac{2b_n}{n} \sum_{t=2}^n |f_n(t)-f_n(t-1)|^2 \leq \frac{2b_n}{n} v_n,
    \end{align*}
    which tends to zero by assumption.

    To handle $\widetilde{\Sigma}_n(u)$, we decompose
    \begin{align*}
        \widetilde{\Sigma}_n(u) 
        &= \frac{1}{2n}\sum_{t=b_n}^{\lfloor un\rfloor -b_n}  \left( \frac{1}{\sqrt{b_n}}\sum_{i=0}^{b_n-1} \eta_{t-i}  \right)^2 
        + \frac{1}{2n}\sum_{t=b_n}^{\lfloor un\rfloor -b_n}  \left( \frac{1}{\sqrt{b_n}}\sum_{i=0}^{b_n-1} \eta_{t+i}  \right)^2 \\
        &\quad + \frac{1}{n}\sum_{t=b_n}^{\lfloor un\rfloor -b_n}  \left( \frac{1}{\sqrt{b_n}}\sum_{i=0}^{b_n-1} \eta_{t-i}  \right)\left( \frac{1}{\sqrt{b_n}}\sum_{i=0}^{b_n-1} \eta_{t+i}  \right) \\
        &= \tfrac{1}{2}\widetilde{\Sigma}^-_n(u) + \tfrac{1}{2}\widetilde{\Sigma}^+_n(u) + \widetilde{\Sigma}^\pm_n(u).
    \end{align*}
    Theorem 5.1 of \cite{mies_sequential_2023} shows that in the regime $1\ll b_n\ll n$, and $\beta\geq 3$, 
    \begin{align*}
        \sup_{u\in[0,1]}\left|\widetilde{\Sigma}_n^-(u) - \frac{1}{n}\sum_{t=1}^{\lfloor un\rfloor -b_n} \sigma^2_{\infty,n}(t) \right| = \mathcal{O}_P\left( \sqrt{\tfrac{b_n}{n}} + \tfrac{1}{b_n} \right),\\ 
        \text{where} \quad \sigma^2_{\infty,n}(t) = \sum_{h=-\infty}^\infty \Cov(G_{t,n}(\beps_0), G_{t,n}(\beps_h)).
    \end{align*}
    Assumption \eqref{ass:LS3} implies that $\sigma^2_{\infty,n}(t)$ is bounded, see \cite[Prop.~1]{mies_strong_2024}, and thus $\frac{1}{n}\sum_{t=1}^{\lfloor un\rfloor -b_n} \sigma^2_{\infty,n}(t) = \int_0^u \sigma_{\infty,n}^2(\lfloor vn\rfloor)\, dv + \mathcal{O}(b_n/n)$. 
    As shown in the proof of \cite[Lemma 4]{mies_strong_2024}, Assumptions \eqref{ass:LS1} and \eqref{ass:LS3} imply the convergence $\int_0^u \sigma_{\infty,n}^2(\lfloor vn\rfloor)\, dv \to \int_0^u \sigma_{\infty}^2(v)\, dv$, and thus $\widehat{\Sigma}_n^-(u)\pconv \Sigma(u)$. 
    Upon an index shift, the same arguments apply to $\widetilde{\Sigma}_n^+(u)$.

    To show that $\widetilde{\Sigma}_n^\pm(u)\to 0$, we write $\widetilde{\Sigma}_n^\pm(u)=\frac{1}{n}\sum_{t=b_n}^{n -b_n} \chi_{t+b_n,n}$, with addends given by $\chi_{t+b_n,n}=\frac{1}{b_n} \sum_{i,j=0}^{b_n-1} \eta_{t-i}\eta_{t+j}$. 
    Since $\eta_t=G_{t,n}(\beps_t)$, we may also write $\chi_{t,n} = H_{t,n}(\beps_t)$ for some kernel $H_{t,n}$. 
    Proceeding as in the proof of \cite[Thm.~5.1]{mies_sequential_2023}, in particular equations (30) and (31), we can derive that 
    \begin{align*}
        \|H_{t,n}(\beps_t) - H_{t,n}(\beps_{t,h})\|_{L_2} \leq K \frac{[(h-b_n)\vee 1]^{1-\beta}}{\sqrt{b_n}}
    \end{align*}
    for some constant $K$, and
    \begin{align*}
        \left\|\max_{k=1,\ldots, n} \sum_{t=1}^k (\chi_{t,n} - \E \chi_{t,n}) \right\|_{L_2} = \mathcal{O}\left(\sqrt{n b_n}\right).
    \end{align*}
    Hence, $\sup_{u} |\widetilde{\Sigma}^\pm_n(u)-\E\widetilde{\Sigma}^\pm_n(u)| = \mathcal{O}_P(\sqrt{b_n/n})$. 
    Moreover, 
    \begin{align*}
        \E \chi_{t+b_n, n}  
        &\leq  \frac{1}{b_n} \sum_{i,j=0}^{b_n-1} |\Cov(\eta_{t-i}, \eta_{t+j})| \\
        &\leq  \frac{1}{b_n} \sum_{i,j=0}^{b_n-1} (i+j+1)^{-\beta} \\
        &\leq  \frac{K}{n}\sum_{t=b_n}^{n -b_n} \frac{1}{b_n} \sum_{i=0}^{b_n-1} (i+1)^{1-\beta} 
        \quad \leq K b_n^{1-\beta},
    \end{align*}
    see \cite[Prop.~1]{mies_strong_2024} for the bound on the autocovariance of $\eta_t$. 
    Thus, $\sup_u \E(\widetilde{\Sigma}_n^\pm(u)) = \mathcal{O}(b_n^{1-\beta})$. 
    Jointly, we find
    \begin{align*}
        \sup_{u\in[0,1]} \left| \widetilde{\Sigma}_n(u)-\frac{1}{n} \sum_{t=1}^{\lfloor un\rfloor -b_n} \sigma^2_{\infty,n}(t) \right| = \mathcal{O}_P\left( \sqrt{\frac{b_n}{n}} + b_n^{-1} + b_n^{1-\beta} \right) = \mathcal{O}_P\left( \sqrt{\frac{b_n}{n}} + b_n^{-1} \right),
    \end{align*}
    since $\beta>3$, and in combination with \eqref{ass:LS3},
    \begin{align*}
        \widetilde{\Sigma}_n(u) - \Sigma(u) \pconv 0.
    \end{align*}
    The latter convergence holds uniformly by monotonicity and continuity of the limit $\Sigma(u)$, though without a rate, since \eqref{ass:LS3} does not state a rate.
    Moreover, together with \eqref{eqn:intvar-err}, we note that
    \begin{align*}
        \sup_{u\in[0,1]} \left| \widehat{\Sigma}_n(u)-\frac{1}{n} \sum_{t=1}^{\lfloor un\rfloor -b_n} \sigma^2_{\infty,n}(t) \right| = \mathcal{O}_P\left( \sqrt{\frac{v_n b_n}{n}} + b_n^{-1} \right).
    \end{align*}
    Since $\sigma^2_{\infty,n}(t)$ is bounded, we find that 
    \begin{align*}
        \sup_{|u-v|>c_n} \left| \widehat{\Sigma}_n(u) - \widehat{\Sigma}_n(v) \right|  
        =\mathcal{O}_P\left( \sqrt{\frac{v_n b_n}{n}} + b_n^{-1} + c_n \right).
    \end{align*}
\end{proof}
    
\begin{proof}[Proof of Theorem \ref{thm:changepoint-LS}]
    In the remainder of this proof, all probabilities only concern the randomness of $W$, i.e.\ we work conditionally on $\widehat{\Sigma}_n$.
For any $0\leq a < b\leq 1$, introduce
\begin{align*}
     Z_n(a,b) &= \sup_{|u-v|\in[a,b]} \frac{|B(\widehat{\Sigma}_n(u)) - B(\widehat{\Sigma}_n(v))|}{\rho_2(|u-v|)} \\
     Z(a,b) &= \sup_{|u-v|\in[a,b]} \frac{|B(\Sigma(u)) - B(\Sigma(v))|}{\rho_2(|u-v|)}. 
\end{align*}
The uniform consistency of $\widehat{\Sigma}_n(u)$ implies that $Z_n(c,1)\wconv Z(c,1)$ for any fixed $c\in(0,1)$.
Moreover, $Z(c,1)\wconv Z(0,1)$ as $c\to 0$. 
Using standard arguments \citep[Thm.~3.2]{Billingsley1999}, we obtain the thresholded weak convergence $(Z_n(c_n,1)\vee \tau)\wconv (Z(0,1)\vee \tau)$ if we can show that, for any $\epsilon>0$, 
\begin{align}
    0 &\overset{!}{=} \lim_{c\to 0} \limsup_{n\to\infty} P\left( \left|(Z_n(c_n,1)\vee \tau)- (Z_n(c,1)\vee \tau)\right| >\epsilon \right) \label{eqn:LS-limsup} \\
    &\leq \lim_{c\to 0} \limsup_{n\to\infty} P\left( Z_n(c_n,c) >\tau \right). \nonumber
\end{align}
To this end, observe that
\begin{align*}
    Z_n(c_n,c) &= \sup_{|u-v|\in[c_n,c]} \frac{|B(\widehat{\Sigma}_n(u)) - B(\widehat{\Sigma}_n(v))|}{\rho_2(|u-v|)}  \\
    &\leq \sup_{|u-v|\in[0,c]} \frac{|B(\widehat{\Sigma}_n(u)) - B(\widehat{\Sigma}_n(v))|}{\rho_2(|\widehat{\Sigma}_n(u)-\widehat{\Sigma}(v)|)} \;\cdot\; \sup_{|u-v|\geq [c_n,1]} \frac{\rho_2(|\widehat{\Sigma}_n(u)-\widehat{\Sigma}_n(v)|)}{\rho_2(|u-v|)}\\
    &\leq \sup_{\substack{|u-v|\in[0,\hat{c}] \\ u,v\leq \widehat{\Sigma}_n(1)}} \frac{|B(u) - B(v)|}{\rho_2(|u-v|)} \;\cdot\; \sup_{u,v\in[0,1]} \frac{\rho_2(\widehat{R}|u-v|)}{\rho_2(|u-v|)}, 
\end{align*}
for the random variables
\begin{align*}
    \hat{c} &= \sup_{|u-v|\leq c} \left|\widehat{\Sigma}_n(u) - \widehat{\Sigma}_n(v)\right| \leq Kc + \mathcal{O}_P\left(  \sqrt{\frac{v_n b_n}{n }} + \frac{1}{b_n}\right), \\
    \widehat{R} &= \sup_{|u-v|\geq c_n} \frac{|\widehat{\Sigma}_n(u)-\widehat{\Sigma}_n(v)|}{|u-v|} \leq K + \mathcal{O}_P\left(\sqrt{\frac{v_n b_n}{n c_n^2}} + \frac{1}{b_n c_n}\right).
\end{align*}
The $\mathcal{O}_P(\dots)$ terms vanish by virtue of our rate assumptions.
Lévy's Theorem on the modulus of continuity of $B$ yields, for any $\epsilon>0$
\begin{align*}
    &\lim_{c\to\infty} \limsup_{n\to\infty} P\left(\sup_{\substack{|u-v|\in[0,\hat{c}] \\ u,v\leq \widehat{\Sigma}_n(1)}} \frac{|B(u) - B(v)|}{\rho_2(|u-v|)}>\sqrt{2}(1+\epsilon)\right) \\
    &\leq \lim_{c\to\infty} P\left(\sup_{\substack{|u-v|\in[0,2Kc] \\ u,v\leq 2\Sigma(1)}} \frac{|B(u) - B(v)|}{\rho_2(|u-v|)}>\sqrt{2}(1+\epsilon)\right) = 0.
\end{align*}
Moreover, by Lemma \ref{Lemma:SchillVarianteWeibull}, the function $\zeta(R)=\sup_{u,v\in[0,1]} \frac{\rho_2(R|u-v|)}{\rho_2(|u-v|)}$ is well-defined and strictly increasing. 
Since $P(\widehat{R}>2K)\to 0$, we find that 
\begin{align*}
    P\left(\sup_{u,v\in[0,1]} \frac{\rho_2(\widehat{R}|u-v|)}{\rho_2(|u-v|)} > \zeta(2K) \right) \to 0.
\end{align*}
This yields
\begin{align*}
    \lim_{c\to\infty} \limsup_{n\to\infty} P\left( \sup_{\substack{|u-v|\in[0,\hat{c}] \\ u,v\leq \widehat{\Sigma}_n(1)}} \frac{|B(u) - B(v)|}{\rho_2(|u-v|)} \; \sup_{u,v\in[0,1]} \frac{\rho_2(\widehat{R}|u-v|)}{\rho_2(|u-v|)} > \sqrt{2}(1+\epsilon)\cdot \zeta(2K) \right) = 0,
\end{align*}
which establishes \eqref{eqn:LS-limsup} for any $\tau>\sqrt{2} \zeta(2K)$.
Thus, we find for any $t>\tau$,
\begin{align*}
    P\left( |B(\widehat{\Sigma}_n(u))|_{\rho_2} > t \,\Big|\, \widehat{\Sigma}_n \right) \pconv P\left( |B(\Sigma(u))|_{\rho_2} > t \right).
\end{align*}
Note that this indeed holds for all $t>\tau$, since the limit has a continuous distribution \cite{lifshits_absolute_1984}.
This limiting continuity also implies the convergence of quantiles.
\end{proof}

\subsection*{Proof of Proposition \ref{prop:sparse}}
    Consider the signal detection problem first. 
    Let $c_n = \frac{a_n}{n}$ and $d_n = \frac{a_n+l_n}{n}$. 
    By assumption \ref{eqn:condition-grid}, there exists $(u_n, v_n)\in \mathcal{G}$ such that $u_n \leq c_n < d_n\leq v_n$ and $|u_n-v_n| \leq K \frac{l_n}{n}$.
    Then
    \begin{align*}
        |S_n|_{\rho_2,\mathcal{G}} &\geq \frac{|S_n(v_n)-S_n(u_n)|}{\rho_2(K|d_n-c_n|)}.
    \end{align*}
    Similar to the proof of Theorem \ref{thm:signal}, we have $|\E(S_n(v_n)-S_n(u_n))| \geq \sqrt{n} |d_n-c_n| |\mu_n|$ and $\Var(S_n(v_n)-S_n(u_n)) = \sqrt{|v_n-u_n|} \leq \sqrt{K|d_n-c_n|}$, which yields 
    \begin{align*}
        \frac{|S_n(v_n)-S_n(u_n)|}{\rho_2(K|d_n-c_n|)} &\asymp \frac{\sqrt{n}|\mu_n| \sqrt{|d_n-c_n|}}{\sqrt{\log|d_n-c_n| }}\\
        &\asymp \frac{|\mu_n| \sqrt{l_n}}{\sqrt{\log (n/l_n)}},
    \end{align*}
    which tends to infinity by assumption.

    Now consider the change localization problem.
    Denote $w_k = w(\tau_k,z)$ for $\tau_k\in \mathcal{D}(z)$, as well as $c_n = \frac{\tau_k - \frac{w_k}{4K}}{n}$ and $d_n=\frac{\tau_k + \frac{w_k}{4K}}{n}$.
    By assumption \eqref{eqn:condition-grid}, there are $(u_n, v_n)\in \mathcal{G}$ such that $u_n \leq a_n < b_n\leq v_n$ and $|u_n-v_n| \leq K |a_n-b_n| \leq  w_k\leq L_k$.
    As in the proof of Theorem \ref{thm:changepoint}, we find for $J_n(\tau_k) = [ n u_n, n v_n]$ that
    \begin{align*}
        T_n^*\left(\Fhyp_{\text{const}}^{J_n(\tau_k)} \right)
        &= \inf_{a\in\R} \sup_{\substack{[un, vn] \subset J_n(\tau_k) \\ (u,v)\in\mathcal{G}}} \frac{|S_n(v)-S_n(u)-[v-u]a|}{\rho_2(v-u)}\\
        &\geq \frac{\frac{w_k \delta_k}{4K\sqrt{n}}}{\rho_2(w_k/n)} - |\widetilde{S}_n|_{\rho_2} 
        \qquad \geq c\sqrt{z} - |\widetilde{S}_n|_{\rho_2},
    \end{align*}
    which tends to infinity as $z\to \infty$, while $|\widetilde{S}_n|_{\rho_2}$ remains stochastically bounded.
    Since $|J_n(\tau_k)|\leq w(\tau_k,z)$, this completes the proof.

\newpage
\bibliography{literature.bib}

\begin{thebibliography}{57}
\expandafter\ifx\csname natexlab\endcsname\relax\def\natexlab#1{#1}\fi

\bibitem[{Arias-Castro et~al.(2011)Arias-Castro, Candès \& Durand}]{arias-castro_detection_2011}
\textsc{Arias-Castro, E.}, \textsc{Candès, E.~J.} \& \textsc{Durand, A.} (2011).
\newblock Detection of an anomalous cluster in a network.
\newblock \textit{The Annals of Statistics} \textbf{39}.

\bibitem[{Bastian \& Dette(2025)}]{bastian_multiscale_2025}
\textsc{Bastian, P.} \& \textsc{Dette, H.} (2025).
\newblock Multiscale detection of practically significant changes in a gradually varying time series.
\newblock ArXiv:2504.15872 [stat].

\bibitem[{Billingsley(1999)}]{Billingsley1999}
\textsc{Billingsley, P.} (1999).
\newblock \textit{Convergence of {Probability} {Measures}}.
\newblock John Wiley \& Sons.

\bibitem[{Chan \& Walther(2013)}]{chan_detection_2013}
\textsc{Chan, H.~P.} \& \textsc{Walther, G.} (2013).
\newblock Detection with the scan and the average likelihood ratio.
\newblock \textit{Statistica Sinica} , 409--428Publisher: JSTOR.

\bibitem[{Chatterjee \& Lafferty(2019)}]{chatterjee_adaptive_2019}
\textsc{Chatterjee, S.} \& \textsc{Lafferty, J.} (2019).
\newblock Adaptive risk bounds in unimodal regression.
\newblock \textit{Bernoulli} \textbf{25}.

\bibitem[{Cho \& Kirch(2021)}]{cho_data_2021}
\textsc{Cho, H.} \& \textsc{Kirch, C.} (2021).
\newblock Data segmentation algorithms: {Univariate} mean change and beyond.
\newblock \textit{Econometrics and Statistics} , S2452306221001234.

\bibitem[{Cui et~al.(2021)Cui, Levine \& Zhou}]{cui_estimation_2021}
\textsc{Cui, Y.}, \textsc{Levine, M.} \& \textsc{Zhou, Z.} (2021).
\newblock Estimation and inference of time-varying auto-covariance under complex trend: {A} difference-based approach.
\newblock \textit{Electronic Journal of Statistics} \textbf{15}.

\bibitem[{Dahlhaus(1997)}]{Dahlhaus1997}
\textsc{Dahlhaus, R.} (1997).
\newblock Fitting time series models to nonstationary processes.
\newblock \textit{Annals of Statistics} \textbf{25}, 1--37.

\bibitem[{Dahlhaus et~al.(2019)Dahlhaus, Richter \& Wu}]{Dahlhaus2017}
\textsc{Dahlhaus, R.}, \textsc{Richter, S.} \& \textsc{Wu, W.~B.} (2019).
\newblock Towards a general theory for nonlinear locally stationary processes.
\newblock \textit{Bernoulli} \textbf{25}, 1013--1044.

\bibitem[{Dette et~al.(2020)Dette, Eckle \& Vetter}]{dette_multiscale_2020}
\textsc{Dette, H.}, \textsc{Eckle, T.} \& \textsc{Vetter, M.} (2020).
\newblock Multiscale change point detection for dependent data.
\newblock \textit{Scandinavian Journal of Statistics} \textbf{47}, 1243--1274.

\bibitem[{Dette et~al.(2019)Dette, Wu \& Zhou}]{Dette2018}
\textsc{Dette, H.}, \textsc{Wu, W.} \& \textsc{Zhou, Z.} (2019).
\newblock Change {Point} {Analysis} of {Correlation} in {Non}-stationary {Time} {Series}.
\newblock \textit{Statistica Sinica} \textbf{29}, 611--643.
\newblock ArXiv: 1801.10478.

\bibitem[{Dümbgen \& Spokoiny(2001)}]{dumbgen_multiscale_2001}
\textsc{Dümbgen, L.} \& \textsc{Spokoiny, V.~G.} (2001).
\newblock Multiscale {Testing} of {Qualitative} {Hypotheses}.
\newblock \textit{The Annals of Statistics} \textbf{29}.

\bibitem[{Dümbgen \& Walther(2008)}]{dumbgen_multiscale_2008}
\textsc{Dümbgen, L.} \& \textsc{Walther, G.} (2008).
\newblock Multiscale inference about a density.
\newblock \textit{The Annals of Statistics} \textbf{36}.

\bibitem[{ENTSO-E(2025)}]{spain2025entsoe}
\textsc{ENTSO-E} (2025).
\newblock {ENTSO-E expert panel initiates the investigation into the causes of Iberian blackout}.
\newblock https://www.entsoe.eu/news/2025/05/09/entso-e-expert-panel-initiates-the-investigation-into-the-causes-of-iberian-blackout/.
\newblock Retrieved on 2025-05-23.

\bibitem[{Feng et~al.(2022)Feng, Chen, Han, Carroll \& Samworth}]{feng_nonparametric_2022}
\textsc{Feng, O.~Y.}, \textsc{Chen, Y.}, \textsc{Han, Q.}, \textsc{Carroll, R.~J.} \& \textsc{Samworth, R.~J.} (2022).
\newblock Nonparametric, tuning‐free estimation of {S}‐shaped functions.
\newblock \textit{Journal of the Royal Statistical Society: Series B (Statistical Methodology)} \textbf{84}, 1324--1352.

\bibitem[{Frick et~al.(2014)Frick, Munk \& Sieling}]{frick2014multiscale}
\textsc{Frick, K.}, \textsc{Munk, A.} \& \textsc{Sieling, H.} (2014).
\newblock Multiscale change point inference.
\newblock \textit{Journal of the Royal Statistical Society: Series B: Statistical Methodology} , 495--580Publisher: JSTOR.

\bibitem[{Frigo et~al.(2019)Frigo, Derviskadic, Zuo \& Paolone}]{frigo_pmu-based_2019}
\textsc{Frigo, G.}, \textsc{Derviskadic, A.}, \textsc{Zuo, Y.} \& \textsc{Paolone, M.} (2019).
\newblock {PMU}-{Based} {ROCOF} {Measurements}: {Uncertainty} {Limits} and {Metrological} {Significance} in {Power} {System} {Applications}.
\newblock \textit{IEEE Transactions on Instrumentation and Measurement} \textbf{68}, 3810--3822.

\bibitem[{Fryzlewicz(2014)}]{fryzlewicz_wild_2014}
\textsc{Fryzlewicz, P.} (2014).
\newblock Wild binary segmentation for multiple change-point detection.
\newblock \textit{The Annals of Statistics} \textbf{42}.

\bibitem[{Fryzlewicz(2024{\natexlab{a}})}]{fryzlewicz_narrowest_2024}
\textsc{Fryzlewicz, P.} (2024{\natexlab{a}}).
\newblock Narrowest {Significance} {Pursuit}: inference for multiple change-points in linear models.
\newblock \textit{Journal of American Statistical Association} \textbf{119}, 1633--1646.
\newblock ArXiv: 2009.05431.

\bibitem[{Fryzlewicz(2024{\natexlab{b}})}]{fryzlewicz_robust_2024}
\textsc{Fryzlewicz, P.} (2024{\natexlab{b}}).
\newblock Robust {Narrowest} {Significance} {Pursuit}: inference for multiple change-points in the median.
\newblock \textit{Journal of Business \& Economic Statistics} \textbf{42}, 1389--1402.
\newblock ArXiv:2109.02487 [stat].

\bibitem[{Gavioli-Akilagun \& Fryzlewicz(2025)}]{gavioli-akilagun_fast_2025}
\textsc{Gavioli-Akilagun, S.} \& \textsc{Fryzlewicz, P.} (2025).
\newblock Fast and optimal inference for change points in piecewise polynomials via differencing.
\newblock \textit{Electronic Journal of Statistics} \textbf{19}.

\bibitem[{Glaz et~al.(2009)Glaz, Pozdnyakov \& Wallenstein}]{glaz_scan_2009}
\textsc{Glaz, J.}, \textsc{Pozdnyakov, V.} \& \textsc{Wallenstein, S.}, eds. (2009).
\newblock \textit{Scan {Statistics}: {Methods} and {Applications}}.
\newblock Boston, MA: Birkhäuser Boston.

\bibitem[{Hamadouche(1998)}]{hamadouche_weak_1998}
\textsc{Hamadouche, D.} (1998).
\newblock Weak convergence of smoothed empirical process in {Hölder} spaces.
\newblock \textit{Statistics \& Probability Letters} \textbf{36}, 393--400.

\bibitem[{Hamadouche(2000)}]{hamadouche_invariance_2000}
\textsc{Hamadouche, D.} (2000).
\newblock Invariance principles in {Holder} spaces.
\newblock \textit{Portugaliae Mathematica} \textbf{57}, 127--152.
\newblock Publisher: Lisboa, Gazeta de matematica [etc.].

\bibitem[{Khismatullina \& Vogt(2020)}]{khismatullina_multiscale_2020}
\textsc{Khismatullina, M.} \& \textsc{Vogt, M.} (2020).
\newblock Multiscale {Inference} and {Long}-{Run} {Variance} {Estimation} in {Non}-{Parametric} {Regression} with {Time} {Series} {Errors}.
\newblock \textit{Journal of the Royal Statistical Society Series B: Statistical Methodology} \textbf{82}, 5--37.

\bibitem[{König et~al.(2020)König, Munk \& Werner}]{konig_multidimensional_2020}
\textsc{König, C.}, \textsc{Munk, A.} \& \textsc{Werner, F.} (2020).
\newblock Multidimensional multiscale scanning in exponential families: {Limit} theory and statistical consequences.
\newblock \textit{The Annals of Statistics} \textbf{48}.

\bibitem[{Lamperti(1962)}]{lamperti_convergence_1962}
\textsc{Lamperti, J.} (1962).
\newblock On convergence of stochastic processes.
\newblock \textit{Transactions of the American Mathematical Society} \textbf{104}, 430--435.

\bibitem[{Lifshits(1984)}]{lifshits_absolute_1984}
\textsc{Lifshits, M.~A.} (1984).
\newblock Absolute continuity of functionals of “supremum” type for {Gaussian} processes.
\newblock \textit{Journal of Soviet Mathematics} \textbf{27}, 3103--3112.

\bibitem[{Liu et~al.(2013)Liu, Xiao \& Wu}]{Liu2013}
\textsc{Liu, W.}, \textsc{Xiao, H.} \& \textsc{Wu, W.~B.} (2013).
\newblock Probability and moment inequalities under dependence.
\newblock \textit{Statistica Sinica} \textbf{23}, 1257--1272.

\bibitem[{Machowski et~al.(2020)Machowski, Bialek \& Bumby}]{machowski_power_2020}
\textsc{Machowski, J.}, \textsc{Bialek, J.~W.} \& \textsc{Bumby, J.~R.} (2020).
\newblock \textit{Power system dynamics, stability and control}.
\newblock Chichester, G. B: Wiley, 3rd ed.

\bibitem[{Mies(2023)}]{mies_functional_2023}
\textsc{Mies, F.} (2023).
\newblock Functional estimation and change detection for nonstationary time series.
\newblock \textit{Journal of the American Statistical Association} \textbf{118}, 1011--1022.

\bibitem[{Mies(2024)}]{mies_strong_2024}
\textsc{Mies, F.} (2024).
\newblock Strong {Gaussian} approximations with random multipliers.
\newblock ArXiv:2412.14346 [math].

\bibitem[{Mies \& Steland(2023)}]{mies_sequential_2023}
\textsc{Mies, F.} \& \textsc{Steland, A.} (2023).
\newblock Sequential {Gaussian} approximation for nonstationary time series in high dimensions.
\newblock \textit{Bernoulli} \textbf{29}, 3114--3140.

\bibitem[{Niu \& Zhang(2012)}]{niu_screening_2012}
\textsc{Niu, Y.~S.} \& \textsc{Zhang, H.} (2012).
\newblock The screening and ranking algorithm to detect {DNA} copy number variations.
\newblock \textit{The Annals of Applied Statistics} \textbf{6}.

\bibitem[{Pein et~al.(2017)Pein, Sieling \& Munk}]{pein_heterogeneous_2017}
\textsc{Pein, F.}, \textsc{Sieling, H.} \& \textsc{Munk, A.} (2017).
\newblock Heterogeneous change point inference.
\newblock \textit{Journal of the Royal Statistical Society. Series B (Statistical Methodology)} \textbf{79}, 1207--1227.
\newblock Publisher: [Royal Statistical Society, Wiley].

\bibitem[{Pilliat et~al.(2023)Pilliat, Carpentier \& Verzelen}]{pilliat_optimal_2023}
\textsc{Pilliat, E.}, \textsc{Carpentier, A.} \& \textsc{Verzelen, N.} (2023).
\newblock Optimal multiple change-point detection for high-dimensional data.
\newblock \textit{Electronic Journal of Statistics} \textbf{17}, 1240--1315.
\newblock ArXiv:2011.07818 [math, stat].

\bibitem[{Pinelis(1994)}]{Pinelis1994a}
\textsc{Pinelis, I.} (1994).
\newblock Optimum {Bounds} for the {Distributions} of {Martingales} in {Banach} {Spaces}.
\newblock \textit{The Annals of Probability} \textbf{22}, 1679--1706.

\bibitem[{Rackauskas \& Suquet(2004{\natexlab{a}})}]{rackauskas_holder_2004}
\textsc{Rackauskas, A.} \& \textsc{Suquet, C.} (2004{\natexlab{a}}).
\newblock Hölder norm test statistics for epidemic change.
\newblock \textit{Journal of Statistical Planning and Inference} \textbf{126}, 495--520.

\bibitem[{Rackauskas \& Suquet(2004{\natexlab{b}})}]{rackauskas_necessary_2004-1}
\textsc{Rackauskas, A.} \& \textsc{Suquet, C.} (2004{\natexlab{b}}).
\newblock Necessary and {Sufficient} {Condition} for the {Functional} {Central} {Limit} {Theorem} in {Hölder} {Spaces}.
\newblock \textit{Journal of Theoretical Probability} \textbf{17}, 221--243.

\bibitem[{Rackauskas \& Suquet(2004{\natexlab{c}})}]{rackauskas_necessary_2004}
\textsc{Rackauskas, A.} \& \textsc{Suquet, C.} (2004{\natexlab{c}}).
\newblock Necessary and sufficient condition for the {Lamperti} invariance principle.
\newblock \textit{Theory of Probability and Mathematical Statistics} \textbf{68}, 127--137.

\bibitem[{Rackauskas \& Suquet(2007)}]{rackauskas_estimating_2007}
\textsc{Rackauskas, A.} \& \textsc{Suquet, C.} (2007).
\newblock Estimating a {Changed} {Segment} in a {Sample}.
\newblock \textit{Acta Applicandae Mathematicae} \textbf{97}, 189--210.

\bibitem[{Rackauskas \& Suquet(2009)}]{rackauskas_holderian_2009}
\textsc{Rackauskas, A.} \& \textsc{Suquet, C.} (2009).
\newblock Hölderian invariance principle for {Hilbertian} linear processes.
\newblock \textit{ESAIM: Probability and Statistics} \textbf{13}, 261--275.

\bibitem[{Rackauskas \& Wendler(2020)}]{rackauskas_convergence_2020}
\textsc{Rackauskas, A.} \& \textsc{Wendler, M.} (2020).
\newblock Convergence of {U}-processes in {Hölder} spaces with application to robust detection of a changed segment.
\newblock \textit{Statistical Papers} \textbf{61}, 1409--1435.

\bibitem[{Reeves et~al.(2007)Reeves, Chen, Wang, Lund \& Lu}]{reeves_review_2007}
\textsc{Reeves, J.}, \textsc{Chen, J.}, \textsc{Wang, X.~L.}, \textsc{Lund, R.} \& \textsc{Lu, Q.~Q.} (2007).
\newblock A {Review} and {Comparison} of {Changepoint} {Detection} {Techniques} for {Climate} {Data}.
\newblock \textit{Journal of Applied Meteorology and Climatology} \textbf{46}, 900--915.

\bibitem[{Rivera \& Walther(2013)}]{rivera_optimal_2013}
\textsc{Rivera, C.} \& \textsc{Walther, G.} (2013).
\newblock Optimal detection of a jump in the intensity of a {Poisson} process or in a density with likelihood ratio statistics.
\newblock \textit{Scandinavian Journal of Statistics} \textbf{40}, 752--769.

\bibitem[{Rohde(2008)}]{rohde_adaptive_2008}
\textsc{Rohde, A.} (2008).
\newblock Adaptive goodness-of-fit tests based on signed ranks.
\newblock \textit{The Annals of Statistics} \textbf{36}.

\bibitem[{Schilling(2021)}]{schilling_brownian_2021}
\textsc{Schilling, R.~L.} (2021).
\newblock \textit{Brownian {Motion}: a guide to random processes and stochastic calculus}.
\newblock Berlin ; Boston: De Gruyter, 3rd ed.

\bibitem[{Schmidt-Hieber et~al.(2013)Schmidt-Hieber, Munk \& Dümbgen}]{schmidt-hieber_multiscale_2013}
\textsc{Schmidt-Hieber, J.}, \textsc{Munk, A.} \& \textsc{Dümbgen, L.} (2013).
\newblock Multiscale methods for shape constraints in deconvolution: {Confidence} statements for qualitative features.
\newblock \textit{The Annals of Statistics} \textbf{41}.

\bibitem[{Sharpnack \& Arias-Castro(2016)}]{sharpnack_exact_2016}
\textsc{Sharpnack, J.} \& \textsc{Arias-Castro, E.} (2016).
\newblock Exact asymptotics for the scan statistic and fast alternatives.
\newblock \textit{Electronic Journal of Statistics} \textbf{10}.

\bibitem[{Verzelen et~al.(2023)Verzelen, Fromont, Lerasle \& Reynaud-Bouret}]{verzelen_optimal_2023}
\textsc{Verzelen, N.}, \textsc{Fromont, M.}, \textsc{Lerasle, M.} \& \textsc{Reynaud-Bouret, P.} (2023).
\newblock Optimal {Change}-{Point} {Detection} and {Localization}.
\newblock \textit{The Annals of Statistics} \textbf{51}, 1586--1610.
\newblock ArXiv: 2010.11470.

\bibitem[{Vladimirova et~al.(2020)Vladimirova, Girard, Nguyen \& Arbel}]{vladimirova_subweibull_2020}
\textsc{Vladimirova, M.}, \textsc{Girard, S.}, \textsc{Nguyen, H.} \& \textsc{Arbel, J.} (2020).
\newblock Sub‐{Weibull} distributions: {Generalizing} sub‐{Gaussian} and sub‐{Exponential} properties to heavier tailed distributions.
\newblock \textit{Stat} \textbf{9}, e318.

\bibitem[{Vogt \& Dette(2015)}]{Vogt2015}
\textsc{Vogt, M.} \& \textsc{Dette, H.} (2015).
\newblock Detecting gradual changes in locally stationary processes.
\newblock \textit{Annals of Statistics} \textbf{43}, 713--740.

\bibitem[{Walther \& Perry(2022)}]{walther_calibrating_2022}
\textsc{Walther, G.} \& \textsc{Perry, A.} (2022).
\newblock Calibrating the {Scan} {Statistic}: {Finite} {Sample} {Performance} {Versus} {Asymptotics}.
\newblock \textit{Journal of the Royal Statistical Society Series B: Statistical Methodology} \textbf{84}, 1608--1639.

\bibitem[{Wu(2005)}]{Wu2005}
\textsc{Wu, W.~B.} (2005).
\newblock Nonlinear system theory: {Another} look at dependence.
\newblock \textit{Proceedings of the National Academy of Sciences} \textbf{102}, 14150--14154.

\bibitem[{Wu \& Zhao(2007)}]{wu_inference_2007}
\textsc{Wu, W.~B.} \& \textsc{Zhao, Z.} (2007).
\newblock Inference of trends in time series.
\newblock \textit{Journal of the Royal Statistical Society: Series B (Statistical Methodology)} \textbf{69}, 391--410.

\bibitem[{Wu \& Zhou(2011)}]{Wu2011}
\textsc{Wu, W.~B.} \& \textsc{Zhou, Z.} (2011).
\newblock Gaussian approximations for non-stationary multiple time series.
\newblock \textit{Statistica Sinica} \textbf{21}, 1397--1413.

\bibitem[{Zhou(2013)}]{Zhou2013}
\textsc{Zhou, Z.} (2013).
\newblock Heteroscedasticity and autocorrelation robust structural change detection.
\newblock \textit{Journal of the American Statistical Association} \textbf{108}, 726--740.

\end{thebibliography}
\bibliographystyle{apalike}
\end{document}